\documentclass[english,11pt,reqno]{amsart}

\usepackage{latexsym}
\usepackage[applemac]{inputenc}
\usepackage[all,dvips,arc,curve,frame]{xy}
\xyoption{matrix}

\usepackage{enumerate}
\usepackage{cancel} % to use \xcancel
\usepackage{amsmath}
\usepackage{rotating}
\usepackage{amsfonts}
\usepackage{amsthm}
\usepackage{amssymb}
\usepackage{verbatim}
\usepackage{psfrag}
\usepackage[colorlinks, linktocpage, pdfauthor={J. Blanc \& S. Lamy},
citecolor = magenta, linkcolor = magenta]{hyperref}
\usepackage[all]{hypcap}
\newtheorem{lemm}{Lemma}[section]
\newtheorem{prop}[lemm]{Proposition}

\newtheorem{coro}[lemm]{Corollary}
\newtheorem{theo}[lemm]{Theorem}
\newtheorem{fact}[lemm]{Fact}

\theoremstyle{remark}
\newtheorem{rem}[lemm]{Remark}
\newtheorem{exple}[lemm]{Example}
\newtheorem{setup}[lemm]{Set-up}
\newtheorem*{merci}{Aknowledgements}

\newcommand{\Z}[1]{\mathbb{Z}/#1\mathbb{Z}}

\newcommand{\C}{\mathbb{C}}
\newcommand{\z}{\mathbb{Z}}

\newcommand{\NE}{\mathrm{NE}}

\newcommand{\Pic}[1]{\mathrm{Pic}(#1)}
\newcommand{\p}{\mathbb{P}}

\newcommand{\parent}[1]{\textup{(}#1\textup{)}}
\newcommand{\hilb}{\mathcal{H}^S_{g,d}}

\def\eps{{\varepsilon}}
\def\cpo{{\mathbb{P}^1}}
\def\cpd{{\mathbb{P}^2}}
\def\cpt{{\mathbb{P}^3}}
\def\cpq{{\mathbb{P}^4}}

\def\No{\mathcal{N}}
\def\Ol{\mathcal{O}}

\def\Z{\mathbb{Z}}
\def\A{\mathcal{A}}

\newcommand{\vbad}[1]{\xcancel{#1}}
\newcommand{\bad}[1]{{#1}*}

\begin{document}

\title[Weak Fano threefolds obtained by blowing-up a space curve]{Weak Fano threefolds obtained by blowing-up a space curve and construction of Sarkisov links}
\date{March 2012}
\author{J\'er\'emy Blanc}
\address{Mathematisches Institut \\ 
Universit\"at Basel \\
Rheinsprung 21 \\
CH-4051 Basel \\
Switzerland}
\email{jeremy.blanc@unibas.ch}

\author{St\'ephane Lamy}
\address{Mathematics Institute\\
        University of Warwick \\
        Coventry \\
        United Kingdom }
\curraddr{Institut de Math\'ematiques de Toulouse,
	     Universit\'e Paul Sabatier, 	
	     118 route de Narbonne,
	     31062 Toulouse Cedex 9, France}
\email{slamy@math.univ-toulouse.fr}
\thanks{The first author was supported by the Swiss National Science Foundation grant no PP00P2\_128422/1.}
\thanks{The second author was supported by a Marie Curie IntraEuropean Fellowship, on leave from the Institut Camille Jordan, Universit\'e Lyon 1.}

%%%%%%%%%%%%%%%%%%%%%%%%%%%%%%%%%%%%%%%%%
\begin{abstract}
We characterise smooth curves in $\mathbb{P}^3$ whose blow-up produces a threefold with anticanonical divisor big and nef. 
These are curves $C$ of degree $d$ and genus $g$ lying on a smooth quartic, such that  
(i) $4d-30 \le g\le 14$ or $(g,d) = (19,12)$, 
(ii) there is no $5$-secant line, $9$-secant conic, nor $13$-secant twisted cubic to $C$. 
This generalises the classical similar situation for the blow-up of points in $\mathbb{P}^2$.

We describe then Sarkisov links constructed from these blow-ups, and are able to prove the existence of Sarkisov links which were previously 
only known as numerical possibilities. 
 
\end{abstract}

\maketitle

\section{Introduction}

A classical result in the theory of algebraic surfaces is that 
the blow-up of $d$ points in $\p^2$ gives a \emph{del Pezzo surface} (i.e. a surface $X$ with $-K_X$ ample) if and only if the points satisfy the following conditions: 

\begin{enumerate}[$(i)$]
\item The number $d$ of points is at most $8$;
\item No $3$ are collinear, no $6$ are on the same conic, and no $8$ belong to the same cubic being singular at one the $8$ points.
\end{enumerate}

More generally, the blow-up of the $d$ points gives a \emph{weak del Pezzo surface} (i.e. a surface $X$ with $-K_X$ nef and big) if and only if the points satisfy the following conditions (see \cite[III.2, Theorem 1]{Demazure}): 

\begin{enumerate}[$(i)$]
\item The number $d$ of points is at most $8$;
\item No $4$ are collinear, and no $7$ are on the same conic.
\end{enumerate}

Moreover, in both cases conditions $(ii)$ are open conditions on the set of $d$-uples of points of $\p^2$ for $1\le d\le 8$.

Note that these statements can easily be extended to the blow-up of points in higher dimension: we shall derive in \S \ref{sec:pointsP3} a similar characterisation of weak Fano threefolds arising from the blow-up of $d$ points in $\cpt$.\\ 

In this paper we study the similar but more difficult problem for the blow-up of a curve  in $\p^3$. 
We prove in particular that the blow-up of a smooth irreducible curve $C$ of genus $g$ and degree $d$ in $\cpt$ is a weak-Fano threefold if and only if:
\begin{enumerate}[$(i)$]
\item $4d-30 \le g\le 14$ or $(g,d) = (19,12)$;
\item there is no $5$-secant line, $9$-secant conic, nor $13$-secant twisted cubic to $C$, and $C$ is contained in a smooth quartic.
\end{enumerate}

We obtain in fact a more precise result, which is our main theorem below.
Here we denote by $\hilb$ the Hilbert scheme of smooth irreducible curves of genus $g$ and degree $d$ in $\cpt$. 
We always work over the base field $\C$ of complex numbers.

\begin{theo}[see Table~$\ref{tab:bignefmovable}$] \label{thm:gendelpezzo}
Let $C \in \hilb$ be a curve of genus $g$ and degree~$d$. 
We denote by $X$ the blow-up of $\cpt$ along $C$ and by $-K_X$ its anticanonical divisor.
Consider the sets
\begin{align*}
\A_0 &= \{(6,5), (10,6), (8,8), (12,9)\};\\
\A_1 &= \{(0,1), (0,2), (0,3), (0,4), (1,3), (1,4), (2,5), (4,6)\};\\
\A_2 &= \{  (1,5), (3,6), (5,7), (10,9)\}; \\
\A_3 &= \{(0,5), (0,6), (1,6), (2,6), (3,4),\\
   &\hspace{2cm} (3,7), (4,7), (6,7), (6,8), (7,8), (9,8), (9,9), (12,10), (19,12)\}; \\
\A_4 &= \{(0,7), (1,7), (2,7), (2,8), (3,8),\\
   &\hspace{2cm} (4,8), (5,8), (6,9), (7,9), (8,9), (10,10), (11,10), (14,11)\}.
\end{align*}
These pairs form a partition of all $(g,d)$ corresponding to non-empty $\hilb$ with $4d-30 \le g\le 14$ or $(g,d) = (19,12)$, and we have the following characterisations:
\begin{enumerate}[$(i)$]
\item
The variety $X$ is Fano, i.e.\ $-K_X$ is ample, if and only if  
$(g,d) \in \A_1$, or $(g,d) \in \A_2$ and there is no $4$-secant line to $C$;
\item
The variety $X$ is weak Fano, i.e. $-K_X$ is big and nef, if and only if one of the following condition holds:
\begin{itemize}
\item $(g,d) \in \A_1 \cup \{(3,6)\}$;
\item $(g,d) \in  \A_2 \smallsetminus \{(3,6)\}$, and there is no $4$-secant line to $C$; 
\item $(g,d) \in \A_3$, there is no $5$-secant line to $C$, and $C$ is contained in a smooth quartic; 
\item $(g,d) \in \A_4$, there is no $5$-secant line, $9$-secant conic, nor $13$-secant twisted cubic to $C$, and $C$ is contained in a smooth quartic.
\end{itemize}
\end{enumerate}
Moreover, Condition $(i)$ corresponds to a non-empty open subset of $\hilb$ if $(g,d)\in \A_2$, and Condition  $(ii)$ corresponds to a non-empty open subset of $\hilb$ if  
$(g,d)\in \A_3 \cup \A_4$.
Furthermore, for a general curve in these sets, there are finitely many irreducible curves intersecting trivially $K_X$, except for $(g,d)\in \{(3,4)$, $(6,7)$, $(9,8)$, $(12,10)$, $(19,12)\}$.
\end{theo}

\begin{table}[h]
$$\begin{array}{|c|l|l|l|l|}
\hline 
g & \text{plane} &\text{quadric} & \text{cubic} &\text{quartic} \\
\hline 
0& \mathbf{(0,1)}  \mathbf{ (0,2)} & \mathbf{(0,3)}  \mathbf{(0,4)} & (0,5)  (0,6) & (0,7)\\
1& \mathbf{(1,3)} & \mathbf{(1,4)} & \mathbf{(1,5)} (1,6) & (1,7) \\
2 & & \mathbf{(2,5)}&(2,6)&(2,7)  (2,8)\\
3& \bad{(3,4)}&& \mathbf{(3,6)} (3,7)&(3,8)\\
4&& \mathbf{(4,6)}&(4,7)&(4,8)\\
5&&& \mathbf{(5,7)}&(5,8)\\
6& \vbad{(6,5)}&\bad{(6,7)}&(6,8)&(6,9)\\
7&&&(7,8)&(7,9)\\
8&&\vbad{(8,8)}&&(8,9)\\
9&&\bad{(9,8)}&(9,9)&\\
10&\vbad{(10,6)}&& \mathbf{(10,9)}&(10,10)\\
11&&&&(11,10)\\
12 &&\vbad{(12,9)}&\bad{(12,10)}&\\
14 &&&&(14,11)\\ 
\cdots &&&& \\
19 &&& \bad{(19,12)} & \\ \hline
\end{array}$$
\caption{Pairs $(g,d)$ with $(g,d) = (19,12)$ or $4d-30 \le g\le 14$ such that there exists a smooth curve $C\subset \p^3$ of genus $g$ and degree $d$.
The columns correspond to the minimal degree of a surface containing a general curve of type $(g,d)$.
Furthermore, the blow-up $X$ of such a general curve $C\subset \p^3$ has an anticanonical divisor  $-K_X$ which is big; it is ample for bold pairs $\mathbf{(g,d)}$. For the other cases, it is nef except in the four cases crossed out, and the anticanonical morphism is a divisorial contraction if there is a star, and it is a small contraction otherwise.
}
\label{tab:bignefmovable}
\end{table}

Another aspect of our work is the construction of Sarkisov links.
Recently several articles appeared which contain lists of numerical possibilities for Sarkisov links between threefolds.
In a series of two papers \cite{JPR, JPR2}, P. Jahnke, T. Peternell and I. Radloff embark on the classification of smooth threefolds $X$ with big and nef (but not ample) anticanonical divisor, and Picard number equal 2.
With these assumptions there are two contractions on $X$: a Mori contraction (of fibering type or divisorial type) and another one given by the linear system $\lvert -mK_X\rvert$ for $m \gg 0$, which we call the \textit{anticanonical morphism}. 
If the latter one is a small contraction, which is the case treated in \cite{JPR2}, then one can perform a flop to obtain another weak Fano threefold $X'$, with another Mori contraction. 
The whole picture is a Sarkisov link.

There are several cases depending on the type of the Mori contractions on $X$ and $X'$.
In \cite{JPR2} are classified all possible links where at least one of the contractions if of fibering type (conic bundle or del Pezzo fibration).
In the paper \cite{CM} by J. Cutrone and N. Marshburn, the classification is completed with a list of all possible links where both contractions are divisorial.
We should also mention the paper \cite{Kal} by A.-S. Kaloghiros where she produces similar tables with a different motivation in mind.
Note that in many cases the links appearing in these lists are only numerical possibilities, and the actual geometrical existence is left as an open problem.

As a consequence of our study we obtain a proof of the existence of some previously hypothetical links.
We gather in Table  \ref{tab:quarticB&N} the Sarkisov links, involving a flop and starting with the blow-up of a smooth space curve, whose existence was left open: they correspond to the set $\mathcal{A}_4$, which is the fourth column in Table~$\ref{tab:bignefmovable}$, and their existence is proven in Section \ref{sec:existence}.

\begin{table}[ht]
$$\begin{array}{|c|c||l|l|}
\hline
g & d  &  \text{link to} & \text{reference}\\
\hline
0 & 7 & (0,7) \subset \cpt & \text{\cite[90]{CM}} \\
1 & 7 & (1,7) \subset X_{22} & \text{\cite[98]{CM}}\\
2 & 7 & (0,5) \subset V_{4} & \text{\cite[103]{CM}} \\
2 & 8 & (2,8) \subset \cpt  & \text{\cite[49]{CM}}\\
3 & 8 & (3,8) \subset \cpt  & \text{\cite[75]{CM}}\\
4 & 8 & (4,10) \subset V_{5} & \text{\cite[89]{CM}}\\
5 & 8 & (5,8) \subset \cpt & \text{\cite[99]{CM}}\\
6 & 9 & (6,9) \subset \cpt & \text{\cite[50]{CM}}\\
7 & 9 & (0,3) \subset X_{12} & \text{\cite[p. 103]{IP}}\\
8 & 9 & dP5 & \text{\cite[Prop. 6.5(25)]{JPR2}}\\
10 & 10 & (10,10) \subset \cpt  & \text{\cite[51]{CM}}\\
11 & 10 & (11,10) \subset \cpt  & \text{\cite[76]{CM}}\\
14 & 11 & (14,11) \subset\cpt  & \text{\cite[52]{CM}}\\
\hline
\end{array}$$
\caption{Pairs $(g,d)$ such that a general curve in $\hilb$ does not lie on a cubic, and gives rise to a weak Fano threefold $X$ with small anticanonical morphism, hence a Sarkisov link. We indicate also the end product of the link, and a reference where the numerical possibility was announced.}
\label{tab:quarticB&N}
\end{table}

As a final remark, note that the truly new result in this paper lies in the last five lines in Theorem \ref{thm:gendelpezzo}, about the generic existence of these curves.
Indeed it would be possible to recover the possible $(g,d)$ from the papers \cite{CM, JPR2} cited above.
However we shall derive the bound on $g$ and $d$ starting from scratch, and it is interesting to keep in mind that both conditions (i) and (ii) are improvement on the following easy estimates.
Once we know (by Riemann-Roch's formula, see \S \ref{sec:existence}) that the curve $C$ is contained in many (in particular, two distinct) quartic surfaces, we have $d \le 16$. Then condition (i) is really about bringing down this condition to $d \le 12$.
Similarly, a curve $\Gamma$ of degree $n$ and at least $4n + 1$ secant to $C$ would forbid $-K_X$ to be nef; but since by Bezout's Theorem such a bad curve must be contained in all quartic surfaces containing $C$, we have the crude estimate $n \le 16 - d$. Then condition (ii) tells us that it is sufficient to check for the absence of bad curve in degree $n \le 3$, and only among smooth rational ones (and in fact in many cases $n \le 2$ or even $1$ is sufficient, see Proposition \ref{pro:antinef} for details).\\

The organisation of the paper is as follows.

After gathering in Section \ref{sec:pre} some preliminary results, we study the threefolds obtained by blowing-up a smooth curve contained in a surface $S \subset \cpt$ of small degree. This is motivated by the fact that if the blow-up $X$ of a smooth curve $C\subset \p^3$ is a weak Fano threefold, then $C$ is contained in a (possibly reducible) quartic.

The case $\deg(S) \le 2$ in Section \ref{sec:quad} serves as a warm-up. This is essentially a nice exercise, and also the only case where we can do an exhaustive study: we do not make any assumption on the singularities of $S$ or about the generality of $C$.

The case of a Sarkisov link starting with the blow-up of a curve in a smooth cubic was already found in the literature. In Section \ref{sec:cubic} we give a fairly complete account of this case since we shall use a similar construction in the following harder case, and also generalise it to smooth curves contained in a normal rational cubic.

In Section \ref{sec:quartic} we study the case when $C$ is contained in a quartic surface but not in a cubic, where lies the principal difficulty of the paper. The proof of Theorem \ref{thm:gendelpezzo} is given in \S \ref{sec:proof}, as a direct consequence of the previous study.

Since our motivation comes from the study of the Cremona group of rank~$3$ (remark that eight of the cases in Table \ref{tab:quarticB&N} correspond to birational self-maps of $\cpt$), we focused the discussion on the already substantial case of Sarkisov links starting with the blow-up of a smooth curve in $\cpt$. 
However, one could probably obtain in a similar way the existence of Sarkisov links starting from other Fano threefolds\footnote{After this work was completed, Amad, Cutrone and Marshburn produced a follow-up to the paper \cite{CM} where they obtain the existence of many Sarkisov links. Interestingly their method is quite different from ours: see \cite{ACM}.}.
%: we discuss briefly in Section \ref{sec:X22} the case of a Fano threefold $Y_{22}$ of index 1 and genus 12. ?? Enleve pour le moment!

\begin{merci}
We benefited from many comments and discussions at different stages of this project; we would like to thank in particular Anne-Sophie Kaloghiros, Ivan Pan, Hamid Ahmadinezhad, Hirokazu Nasu, Yuri Prokhorov and Igor Dolgachev.

Thanks also to the referee for his interesting remarks, which helped us to improve the exposition of the paper.
\end{merci}

\section{Preliminaries} \label{sec:pre}

\subsection{Sarkisov links} \label{sec:sarkisov}

Let $C\subset \p^3$ be a smooth curve of genus $g$ and degree $d$ and let $\pi\colon X\to \p^3$ be its blow-up. 
We denote by $E = \pi^{-1}(C)$ the exceptional divisor, and $f$ an exceptional curve.
The Picard group $\Pic{X}$ is generated by $H$ and $E$, where $H$ is the pull-back of a hyperplane. 
Moreover, $\mathrm{N}_1(X)$ is generated by $f$ and by $l$, the pull-back of a line. 
The intersection form on $X$ is given by
$$H\cdot l = 1, \quad E\cdot f = -1, \quad H\cdot f = E\cdot l = 0.$$ 

The  pseudo-effective cone $\overline{\text{Eff}}(X)$ is generated by $E$ and $R$, where $R$ should correspond to the strict transform of surfaces of degree $n$ and passing with multiplicity $m$ through $C$, with $m/n$ maximal. 
It is not clear if the supremum is realised, so $R$ might not be represented by an effective divisor. 

Here are some situations where $R$ is realised by an effective divisor, and where the contraction of $R$ gives rise to a Sarkisov link.

If $-K_X$ is ample, then $R$ is an extremal ray, and since $X$ is smooth the contraction $X \to Y$ of $R$ is either a divisorial contraction ($Y$ is a terminal Fano threefold) or a Mori fibration (conic bundle or del Pezzo fibration).

If $-K_X$ is not ample, but big and nef, and if the anticanonical morphism given by $\lvert -mK_X \rvert$, $m \gg 0$, corresponds to a small birational map, then $R$ corresponds to finitely many curves on $X$ and we can consider the flop $X \dashrightarrow X'$ of these curves. Then $X'$ admits an extremal contraction which again can be of divisorial or fibering type.

In any cases we have a Sarkisov link of type II (from $\cpt$ to $Y$) or I (from $\cpt$ to $X'$) of the form  
\begin{equation}\label{eq:link}
\begin{split}
\xymatrix{
& X \ar@{-->}[r] \ar[dl] & X' \ar[dr] \\
\cpt &&& Y
}
\end{split}
\end{equation}
where $X \dashrightarrow X'$ is a flop or an isomorphism.

Note that we would have a similar discussion replacing $\cpt$ by any smooth prime Fano threefold.
More generally, we have:

\begin{prop} \label{pro:link}
Assume that $X$ is a smooth threefold with Picard number $2$,  big and nef anticanonical divisor, and small anticanonical morphism. 
Then the two contractions on $X$ yield a Sarkisov link $($with or without a flop depending if $-K_X$ is ample$)$.  
\end{prop}
This was the point of view adopted in \cite{JPR2}.

%\subsection{Cone of curves}

The effective cone of curves $\overline{\text{NE}}(X) \subset N_1(X)$ is generated by $f$ and $r$, where $r$ should correspond to the strict transforms of curves of degree $m$ which intersect $C$ in $n$ points, with $n/m$ maximal.
Again it is not clear if $r$ is represented by an effective curve.
% However note that since $\overline{\text{NE}}(X)$ is dual to $\text{Nef}(X)$, $r$ is not represented if and only if $R'$ is not, which implies $R' = R$ ???  \\

We shall use repetitively the following simple observation, which follows from the identity $-K_X = 4H - E$:

\begin{lemm}\label{lem:badcurve}
The anticanonical divisor $-K_X$ is nef if and only if there is no irreducible curve in $\NE(X)$ which is equivalent to $ml-nf$ with $n> 4m$.

Similarly if $-K_X$ is ample then there is no irreducible curve in $\NE(X)$ which is equivalent to $ml-nf$ with $n\ge 4m$.
\end{lemm}

We will see these problematic curves either on $X$, as curves equivalent to $ml-nf$, or on $\p^3$, as curves of degree $m$ passing $n$ times (counted with multiplicity) through $C$.

Notice that if $C$ is contained in a quartic surface $S$, then all curves corresponding to $n> 4m$ must lie in $S$. 
Furthermore the curves we are interested in must lie on quartics, as we shall see in \S \ref{sec:smoothquartics}.

\subsection{Linear system of surfaces containing a curve}

We recall the following standard consequence of the Riemann-Roch formula for curves:

\begin{lemm}\label{lem:hypersurfaces}
Let $C\subset \p^3$ be a smooth curve of genus $g$ and degree $d$.

For any integer $n>\frac{-2+2g}{d}$ the projective dimension of the linear system of hypersurfaces of $\p^3$ of degree $n$ that contain $C$ is at least equal to 
$$\frac{(n+1)(n+2)(n+3)}{6}-nd-2+g.$$

In particular, we have the following:

\begin{center}\begin{tabular}{ll}
If ${2g-2}<d<3+g$,& $C$ is contained in a plane.\\
If $(2g-2)/2<d<(9+g)/2$,&$C$ is contained in a quadric surface.\\
If $(2g-2)/3<d<(19+g)/3$,& $C$ is contained in a cubic surface.\\
\end{tabular}\end{center}
\end{lemm}

\begin{proof}
Let $n\ge 1$ be an integer, and denote by $\phi_n\colon \p^3\to \p(V_n)$ the $n$-th Veronese embedding, where $V_n$ is the vector space of polynomials of degree $n$ in $4$ variables, of dimension $N=(n+1)(n+2)(n+3)/6$. Note that $N=4,10,20$ for $n=1,2,3$.

We compose the embedding $C \hookrightarrow \p^3$, which is given by a linear system of degree $d$, with $\phi_n$ and obtain an embedding  $\phi\colon C \hookrightarrow \p^{N-1}$ of degree $dn$. 
Denote by $D$ a hyperplane section with respect to this embedding. If the projective dimension of the complete linear system $\lvert D\rvert$ of all effective divisors of degree $dn$ is strictly smaller than $N-1$, then $\phi(C)$ is contained in a hyperplane, which amounts to say that  $C\subset \p^3$ is contained in a hypersurface of degree $d$.

Applying Riemann-Roch to the system on the curve, we get $$l(D)-l(K-D)=\deg D+ 1- g,$$
where $l(D) = 1 + \dim \lvert D\rvert$ is the vectorial dimension of $\lvert D\rvert$.

Since we assume $n>\frac{-2+2g}{d}$, we have $\deg(K-D)=-2+2g-nd < 0$. So $D$ is non-special and we have $\dim \lvert D\rvert= nd-g$. In consequence, if $nd-g< N-1$ (or equivalently $d<\frac{N-1+g}{n}$), we find a hypersurface of  degree $d$ which contains the curve. More generally, the vectorial dimension of the linear system of hypersurfaces of degree $d$ which contain the curve is at least equal to $N-nd-1+g$.
% When $n$ is equal to $1,2,3$, we find respectively that $d$ has to be strictly less than $3+g,(9+g)/2,(19+g)/3$ to have a positive dimension. This gives the result.
\end{proof}

% ?? Faire une version avec courbe reductible ?

\subsection{Cube of the anticanonical divisor}

\begin{lemm}
\label{lem:K^3}
Let $C \subset Y$ be a smooth curve of genus $g$ in a smooth threefold, and let $\pi\colon X\to Y$ be the blow-up of $C$.
Then
\begin{align*}
K_X^2\cdot E &= -K_Y\cdot C + 2 - 2g;\\
(-K_X)^3 &= (-K_Y)^3 +2K_Y\cdot C - 2 + 2g.
\end{align*}
In particular, if $Y = \p^3$ and $C \in \hilb$, then
\begin{align*}
K_X^2\cdot E &= 2+4d-2g;\\
(-K_X)^3 &=62-8d+2g;
\end{align*}
and $(-K_X)^3 > 0 \Longleftrightarrow 4d -30 \le g $. 
\end{lemm}
%$$ $$

\begin{proof}
Denoting by $E$ the exceptional divisor, we have $K_X = \pi^*K_Y + E$; hence
\begin{align*}
K_X^2\cdot E &= \pi^*K_Y\cdot(\pi^*K_Y\cdot E) + 2(\pi^*K_Y\cdot E)\cdot E + E^3;\\
K_X^3 &= K_Y^3 + 3\pi^*K_Y\cdot(\pi^*K_Y\cdot E) +3(\pi^*K_Y\cdot E)\cdot E + E^3.
\end{align*}
Furthermore, denoting by $f$ a fibre of the ruled surface $E$ we have (see \cite[Lemma 2.2.14]{IP}) 
\begin{align*}
\pi^*K_Y\cdot E &= (K_Y \cdot C) f; & E\cdot f &= -1;\\
\pi^*K_Y\cdot f &= 0; & E^3 &= K_Y\cdot C + 2 -2g. 
\end{align*}
Hence
\begin{align*}
\pi^*K_Y\cdot(\pi^*K_Y\cdot E) &= 0;\\
(\pi^*K_Y\cdot E)\cdot E &= - K_Y\cdot C;
\end{align*}
and the result follows.
\end{proof}

\subsection{A result of Mori}

We shall use the following result from \cite{Mori84}.

\begin{prop}\label{pro:mori1}
Let $g\ge 0$ and $d \ge 1$ be two integers such that $8g < d^2$.
Then there exists a smooth quartic surface $S \subset \cpt$ containing a smooth irreducible curve of genus $g$ and degree $d$. 
\end{prop}

Moreover Mori gives also a converse statement, that we shall need in the slightly more general setting of a reducible curve.

If $C = \cup C_i$ is a projective (possibly reducible and singular) curve, recall that the Hilbert polynomial $P_C(T)$ has the form $P_C(T) = dT + 1 - g$, where $d$ is the degree of $C$ and where by definition $g$ is the \textit{arithmetic genus} of $C$.
If furthermore $C$ lies inside a smooth surface S, the arithmetic genus  $g$ of $C$ satisfies the adjunction formula $2g - 2 = (K_S + C) \cdot C$.
In particular, if $S$ is a $K3$ surface, $2g-2 = C^2$.

The proof of the following result is taken almost verbatim from \cite{Mori84}, where the case of an irreducible curve is treated.
We reproduce the argument for the convenience of the reader.

\begin{prop}
\label{pro:mori}
Let $S \subset \cpt$ be a smooth quartic surface, and  
let $C = \cup C_i \subset S$ be a curve of degree $d$ and genus $g$.
Assume that $C$ is not a complete intersection of $S$ with another surface. 
Then $8g < d^2$. 
\end{prop}

\begin{proof}
Take $H$ a hyperplane section on $S$. 
We have
\begin{align*}
C\cdot H &= d,\\
H\cdot H &= 4,\\
H\cdot(dH-4C) &= 0.
\end{align*}
Since $C$ is not a complete intersection, $dH-4C \not\equiv 0$.
Hence by the Hodge index Theorem we have 
$$0 > (dH-4C)\cdot (dH-4C) = 4 d^2 - 8d^2 + 16(2g-2). $$
This gives $d^2 - 8g > -8$. To show that $d^2 - 8g > 0$, we have to exclude the cases $d^2 - 8g = -7, -4 $ or 0 (other values are not square modulo 8).

If $d^2 - 8g = 0$, then $d = 4d'$ for some $d' \in \Z$.
Let $E = d'H - C$. Then
\begin{align*}
E\cdot H &= (d'H - C)\cdot H = 4d'-d = 0, \\
E^2 &= {d'}^2 H^2 + C^2 - 2d'H\cdot C = 4 \frac{d^2}{16} + 2g-2 - 2 \frac{d^2}{4} = -2.
\end{align*}
Since a quartic surface has trivial canonical divisor and arithmetic genus 1, the Riemann-Roch formula on $S$ gives:
$$l(E) + l(-E) \ge \frac{E^2}{2} +2 = 1.$$
Hence $E$ or $-E$ is an effective curve, but then $H\cdot E = 0$ contradicts $H$ ample. 

If $d^2 - 8g = -7$, then $d = 2d'-1$ for some $d' \in \Z$.
Let $E = d'H - 2C$. Then
\begin{align*}
E\cdot H &= (d'H - 2C)\cdot H = 4d'-2d = 2, \\
E^2 &= {d'}^2 H^2 + 4C^2 - 4d'H\cdot C = 4 \frac{(d+1)^2}{4} + 8g-8 - 2 d(d+1) = 0.
\end{align*}
The Riemann-Roch formula on $S$ gives:
$$l(E) + l(-E) \ge \frac{E^2}{2} +2 = 2.$$

Since $E\cdot H = 2$, $E$ is effective and has at most two irreducible components.
So $E$ is a line or a (possibly degenerate) conic. 
But this contradicts $E^2 = 0$ (by adjunction, a smooth rational curve on $S$ has square $-2$).

Finally, if   $d^2 - 8g = -4$, then $d = 4d'-1$ for some $d' \in \Z$.
Let $E = d'H - C$. Then again $E\cdot H = 2$, $E^2 = 0$ and we can repeat the previous argument. 
\end{proof}

\subsection{Linkage} \label{sec:linkage}

Let $C$, $C'$ be (possibly singular or reducible) curves in $\cpt$, of respective arithmetic genus and degree $(g,d)$ and $(g',d')$.
We say that there is a \textit{linkage} of type $[n_1,n_2]$ between $C$ and $C'$, 
if $C \cup C' \in \cpt$ is the complete intersection of two surfaces of respective degree $n_1, n_2$.
Then 
$$2g(C\cup C') - 2 = n_1n_2(n_1 + n_2 - 4),$$
and we have the following basic relation (see \cite[Proposition 3.1, (vi)]{PS}) between the genus and degree of $C$ and $C'$ :
$$ g-g' = \frac{n_1 + n_2 - 4}{2}(d-d').$$
Remark that Peskine and Szpiro focus on the case of curves which are arithmetically Cohen-Macaulay, but the formula above does not need this assumption (see \cite[(2.1)]{nasu}).

Assume further that $C, C'$ are contained in a smooth surface $S$. Using adjunction formula on $S$ we have
$$2g(C\cup C')-2= (C+C')(C+C'+K_S)=C(C+K_S)+C'(C'+K_S)+2CC'$$
which yields
$$g(C\cup C') = g + g' - 1 +(C\cdot C')_S$$
Note that the intersection does not depend on the choice of the smooth surface $S$.

%We consider a sequence of blow-ups $\sigma_i$ at points $p_i$ for $i=1,\dots,r$ such that $\sigma_r \circ  \cdots \circ \sigma_1\colon \hat{S}\to S$ is the minimal number of blow-ups necessary to separate the strict transforms of $C$ and $C'$.
%Denote by $m_i, m_i'$ the respective multiplicities of the transforms of $C,C'$ at $p_i$: this only depends on $C$ and $C'$, and not on the particular surface $S$.
%We call $\# (C\cap C') = \sum m_im_i'$ the number of intersection points of $C$ and $C'$. 
%Then we easily verify  the following formula:

An application of linkage theory is the study of the irreducible components of the Hilbert scheme $\hilb$, as illustrated by the proof of the following fact that we shall need in the proof of the existence  of the Sarkisov links listed in Table~$\ref{tab:quarticB&N}$.

\begin{prop} \label{pro:nasu}
Let $(g,d)$ be one of the thirteen cases in Table~{\rm \ref{tab:quarticB&N}}.
Then $\hilb$ is irreducible, except in the case $(14,11)$ where there are exactly two irreducible components, one of them corresponding to curves of bidegree $(3,8)$ on a smooth  quadric surface.
\end{prop}

\begin{proof}
It is known that $\hilb$ is irreducible if $g+3 \le d$ (Ein), or if $g+9 \le 2d \le 22$ (Guffroy, see \cite{Guf} and references therein). 
This gives the result for the twelve first cases.

For $(14,11)$ we now sketch a proof, which was explained to us by H. Nasu, using similar techniques as in \cite{nasu}. 

First assume that $C \in \mathcal{H}^S_{14,11}$ is contained in a cubic surface $S$.
Then:
\begin{itemize}
\item $S$ is not a cone: Otherwise by \cite[Lemma 2.10]{nasu} we would have $d \equiv 0$ or $1 \mod 3$, and in our case $d = 11 \equiv 2 \mod 3$;
\item $S$ is normal: Otherwise the identity (see \cite[(2.7)]{nasu})
$$ g = \frac{(k-1)(2d-3k-2)}2$$
should be satisfied for some $k \ge 0$, which is not the case for $(g,d) = (14,11)$. 
\end{itemize}
Then the curve $C$ is a specialization of a curve in a smooth cubic surface (\cite[Theorem 2.7]{nasu}), and the family of such curves has dimension $g + d + 18 = 43$ (\cite[Proposition 2.5]{nasu}).
But it is a standard fact that any irreducible component of $\hilb$ has dimension greater than $4d$ (\cite[Remark 2.6]{nasu}); in particular any irreducible component of $\mathcal{H}^S_{14,11}$ must have dimension at least $44$.

Assume now that $C$ is contained in a quadric. If $C$ was contained in a quadric cone we would have $g = (a-1)(a+e-1)$ and $d = 2a+e$ for some $a,e \in \mathbb N$ (see last part of  the proof of Proposition \ref{pro:quad}); but there is no such relation for $(g,d) = (14,11)$. Thus $C$ must be a curve of bidegree $(3,8)$ on a smooth quadric. The dimension of such curves is $g + 2d + 8 = 44$, and they form an irreducible component of $\mathcal{H}^S_{14,11}$.

Finally we observe that we can construct $C$ contained in an irreducible quartic by linkage.
Indeed, start from a smooth curve $C' \in \mathcal{H}^S_{2,5}$ contained in a smooth quadric $Q'$: if we identify $Q'$ to $\cpo \times \cpo$, $C'$ is a curve of bidegree $(2,3)$. 
One can show that such a curve is $4$-regular, and applying a theorem of Martin-Deschamps and Perrin \cite[Theorem 3.4]{nasu}, we obtain that a general linkage of type $[4,4]$ yields a smooth curve $C \in \mathcal{H}^S_{14,11}$.
Denote by $W$ the family of curves in $\mathcal{H}^S_{14,11}$ obtained by this process.
Then by  \cite[Lemma 2.2]{nasu}, $W$ has dimension $4(11-5) + 20 = 44$ (here we use the fact that $\dim \mathcal{H}^S_{2,5} = 20$, which can be computed directly or alternatively by using linkages of type $[2,4]$ to twisted cubics, whose parameter family has dimension 12).
% Finally consider $C$ contained in an irreducible quartic $Q_1$. By Lemma \ref{lem:hypersurfaces}, the dimension of the system of quartics through $C$ is at least $3$. 
% Then by taking the complete intersection of $Q_1$ with another quartic $Q_2$ containing $C$, we obtain via a linkage of type $[4,4]$ a curve $C'$ of genus $2$ and degree $5$. 
% If smooth, such a curve is contained in a quadric $Q_3$, and by considering the linkage of type $[2,4]$ defined by $Q_3$ and $Q_2$ we obtain a twisted cubic $C''$.
% Conversely, starting from a twisted cubic $C''$, and choosing a quadric $Q_3$ containing $C''$, we can produce curves $C' \in \mathcal{H}^S_{2,5}$ and $C \in \mathcal{H}^S_{14,11}$ by bilinkage, using general quartics $Q_2$ and $Q_1$ and applying a theorem of Martin-Deschamps and Perrin \cite[Theorem 3.4]{nasu}. 
% Denote by $W$ the family of curves in $\mathcal{H}^S_{14,11}$ obtained by this process, which are arithmetically Cohen-Macaulay since a twisted cubic is.
% Then by  \cite[Lemma 2.2]{nasu}, $W$ has dimension $4(11-3) + 8 = 44$.
The closure of such curves is the second irreducible component of $\mathcal{H}^S_{14,11}$, and contains as a codimension 1 closed subset the curves lying in cubic surfaces.
\end{proof}

\subsection{Existence of smooth quartics containing $C$}
\label{sec:smoothquartics}

\begin{prop}\label{pro:insmoothquartic}
If the blow-up $\pi\colon X \to \cpt$ of a smooth curve $C \subset \cpt$ gives rise to a weak Fano threefold $X$, then $C$ is contained in a pencil of quartics whose general member is smooth.
\end{prop}

\begin{proof}
By the Riemann-Roch formula on a threefold and Kawamata-Viehweg vanishing, $\dim \lvert -K_X \rvert = \frac{(-K_X)^3}{2} + 2$.

Assume first that the linear system $\lvert -K_X \rvert$ is base point free, which implies that a general element $S \in \lvert -K_X \rvert$ is a smooth $K3$ surface. Then the restricted system $\lvert -K_X \rvert_{|E}$ is also base point free. 
Since $-K_X\cdot f = 1$, where $f$ is an exceptional curve for $\pi$, we deduce that $S_{|E}$ is a smooth irreducible curve which is a section of $\pi\colon E \to C$.
In consequence $\pi(S)$, which is a general quartic through $C$, is smooth.

We assume now that the base locus of $\lvert -K_X \rvert$ is non empty.
By a result of Shin \cite[Theorem (0.5)]{shin} (building on previous results of Shokurov and Reid), the base locus $\lvert -K_X \rvert$ is isomorphic to a smooth rational curve $\Gamma$, and the general member $S$ of $\lvert -K_X \rvert$ is smooth away from $\Gamma$, hence globally smooth by Bertini's Theorem (Shin assumes $X$ with canonical singularities, but since in our setting $X$ is smooth we get a more precise result).  
We do not repeat the argument of Shin here, but we want to point out that the main point is a result of Saint-Donat \cite{SD}, which states that a big and nef linear system on a smooth $K3$ surface which is not base point free has the form $kM + \Gamma$, where $k > 0$, $M$ is a smooth elliptic curve, $\Gamma$ is a rational $(-2)$-curve and $M\cdot \Gamma = 1$. 
More specifically Saint-Donat proved this for an ample linear system, and the same property holds in the  big and nef case by \cite[Lemma (2.1)]{shin}.
The result follows by applying Saint Donat's result to the linear system  
$\lvert -K_X \rvert_{|S}$.

% By a result of Reid \cite{Reid:1983} (building on a result of Shokurov), in any case the general member of $\lvert -K_X \rvert$ is a $K3$ surface with at most du Val singularities.
% But since in our situation $X$ is smooth (the result of Reid holds for $X$ Gorenstein with canonical singularities) one can deduce that in fact the general member $S$ of $\lvert -K_X \rvert$ is smooth. 
% This is done in \cite[page 40]{IP} in the case where $X$ is Fano, and the argument given there applies as well in the weak Fano case.
% We do not repeat the details of this argument here, but we want to point out that the main point is a result of Saint-Donat \cite{SD}, which states that a big and nef linear system on a smooth $K3$ surface which is not base point free has the form $kM + \Gamma$, where $k > 0$, $M$ is a smooth elliptic curve, $\Gamma$ is a rational $(-2)$-curve and $M\cdot \Gamma = 1$. 
% More specifically Saint-Donat proved this for an ample linear system, and the same property holds in the  big and nef case (see e.g. \cite{shin}).

We deduce that if a general quartic $\pi(S)$ through $C$ is singular, then the singular locus is a unique ordinary double point $p= \pi(\Gamma)$, where $\Gamma \subset S$ is the base locus of $\lvert -K_X \rvert_{|S}$, and is also a fibre of the exceptional divisor $E$. 
We now derive a contradiction from this situation. 

By Lemma \ref{lem:K^3} we have $S\cdot S\cdot E = (-K_X)^2\cdot E = 2+4d-2g$.
On the other hand $S\cdot S\cdot E = (kM+\Gamma)\cdot E = kM\cdot E - 1$.
Hence the relation
$$M\cdot E = \frac{3+4d-2g}{k}.$$

We also have $1 = S\cdot M = -K_X\cdot M = -(\pi^*(K_\cpt) + E)\cdot M = 4\deg \pi(M) - M\cdot E$. Since $\pi(M)$ is an elliptic curve, we have $\deg \pi(M) \ge 3$, and 
$$ M\cdot E \ge 11.$$
On the other hand, again by Lemma \ref{lem:K^3} we have $4d \le g + 30$ hence $3+4d-2g \le 33 -g$. 
Thus we have $3+4d-2g < 33$, since if $g = 0$ then $4d \le g + 30$ implies $4d \le 28 < 30$.
Since $3+4d-2g$ is odd, we obtain $k=1$. 
Moreover we have a linkage of type $[4,4]$ between $\pi(M)$ and $C$, hence $g - g(M) = 2(d - \deg(M))$ with $\deg(M) = 16 - d$ and $g(M) = 1$.
So we obtain the contradiction $g = 2(2d-16) + 1 = 4d -31 <g$. 
\end{proof}

\subsection{Blow-up of points in $\p^k$} \label{sec:pointsP3}
For sake of completeness, we include here the case of blow-up of points in $\p^k$, $k \ge 3$, which is an easy but nice exercise. It is probably well-known to specialists, but we did not find any reference in the literature.
Note however that Geiser and Bertini like involutions associated with the blow-up of 6 or 7 points in $\cpt$ were classically known (see e.g. \cite[4.6 page 91]{Dol-Cremona}).

\begin{prop}
Let $k \ge 3$.
The blow-up $X$ of $d$ points in $\p^k$ is Fano if and only if $d=1$. 
It is weak Fano is and only if  $k = 3$, $d\le 7$ and the following conditions hold:
\begin{enumerate}[$(1)$]
\item
no $3$ points belong to a same line;
\item
no $5$ belong to the same conic \parent{or equivalently to the same plane};
\item
no $7$ belong to the same twisted cubic.
\end{enumerate}
\end{prop}

\begin{proof}
The blow-up of one single point gives a Mori fibre space $X\to \p^{k-1}$ induced by the lines passing through the blown-up point. It is clearly a Fano variety.

Assume now $d\ge 2$.
The strict transform $l$ of any line passing through two points satisfies $K_X \cdot l = -k-1+2(k-1) = k - 3$. 
Hence $-K_X$ cannot be nef if $k \ge 4$. 
We assume now $k = 3$. If $X$ is weak-Fano, we have $(-K_X)^3=64-8d>0$ so $d\le 7$, and the three conditions above are clearly satisfied. It remains to prove that the conditions imply that $X$ is weak-Fano.
 
Let $\Lambda$ be the linear system of quadrics of $\p^3$ passing through the $d$  blown-up points, which corresponds to the image of  $|-\frac{1}{2}K_X|$.
The fact that $d\le 7$ implies that $\dim \Lambda \ge 9-d\ge 2$. 
Since no $5$ of the points belong to the same plane, a general member of $\Lambda$ is irreducible.
  
We can thus find two distinct irreducible quadrics $Q_1,Q_2$ passing through the $d$ points. If one irreducible curve of $X$ intersects $-K_X$ negatively, it has to come from a curve $C\subset \p^3$ contained in $Q_1$ and $Q_2$, which has thus  degree $n\le 4$. Writing $m_1,\dots,m_d$ the multiplicities of $C$ at the points, we have $\sum_{i=1}^d m_i>2n$.
  
If $n=1$, we find at least three points on $C$, impossible by hypothesis. We get a similar contradiction if $C$ is a conic.
  
If $n=3$, we have  either $7$ points on the curve or $6$ points whose one is singular. The curve cannot lie on a plane, otherwise $6$ points would lie on a plane. And it cannot be a twisted cubic by hypothesis.
   
The last case is when $n=4$ which implies that $C$ is a complete intersection of two quadrics hence has arithmetic genus $g = 1$. In particular $C$ admits at most one double point and $\sum_{i=1}^d m_i \le 8 = 2n$: contradiction.   
\end{proof}

\section{Curves in a plane or a quadric} 
\label{sec:quad}

\begin{prop}\label{pro:quad}
Let $C\subset S \subset \p^3$ be a smooth curve of genus $g$ and degree~$d$ which is contained in a plane, a smooth quadric or a quadric cone $S$. 
Let $\pi\colon X\to \p^3$ be the blow-up of $C$.
We note $\overline{\NE}(X) = \langle f, r\rangle$, where $f$ is a $\pi$-exceptional curve. 
Then $-K_X$ is big, $r=l-nf$ where $l$ is the pull-back of a line of $\p^3$ and $n\in \mathbb{N}$, and  one of the following occurs:
\begin{enumerate}[$(i)$]

\item $-K_X$ is ample; there exists a contraction $X\to Y$ which yields a Sarkisov link $\p^3\dasharrow Y$ or $\p^3\dasharrow X$. Namely we are in one of the following cases, the details of which are discussed in Examples {\rm \ref{ex:plane}(i-iii)}, {\rm \ref{ex:quadric}(i-v)} and {\rm \ref{ex:quadriccone}(i-ii)} respectively:
\begin{itemize}
\item $S$ is a plane: $(g,d) = (0,1), (0,2)$ or $(1,3)$;
\item $S$ is a smooth quadric, and $C$ is not a plane curve: $(g,d) = (0,3), (0,4), (1,4), (2,5)$ or $(4,6)$.
\item $S$ is a quadric cone, and $C$ is not contained in any other quadric: $(g,d) = (2,5)$ or $(4,6)$. 
\end{itemize}

\item $-K_X$ is nef but not ample, then the curves numerically equivalent to $r=4l-f$ cover the surface $S$, the anticanonical morphism is a divisorial contraction and we are in one of the following cases:
\begin{itemize}
\item $S$ is a plane: $(g,d) = (3,4)$;
\item $S$ is a smooth quadric, and $C$ is not a plane curve: $(g,d) = (0,5)$, $(3,6)$, $(6,7)$ or $(9,8)$.
\item $S$ is a quadric cone, and $C$ is not contained in any other quadric: $(g,d) = (6,7)$ or $(9,8)$. 
\end{itemize}

\item $-K_X$ is not nef,  then the curves numerically equivalent to $r=l-nf$ for $n\ge 5$ cover the surface $S$, $\lvert -mK_X \rvert$ has a fixed component for any $m$ and we are in one of the following cases:
\begin{itemize}
\item $S$ is a plane: $g=\frac{(d-1)(d-2)}{2}$ and  $d\ge 5$; if $g \le 14$ then $(g,d) = (6,5)$ or $(10,6)$;
\item $S$ is a smooth quadric, and $C$ is not a plane curve; if $4d-30\le g\le 14$ then $(g,d)\in \{ (0, 6)$, $(0, 7)$, $(4, 7)$, $(5, 8)$, $(6, 9)$, $(8, 8)$, $(10, 9)$, $(12, 9)$, $(12, 10)$, $(14, 11)\}$.
\item $S$ is a quadric cone, $C$ is not contained in any other quadric and $g\ge 16$. 
\end{itemize}
\end{enumerate}
Moreover, any of the possibilities given above occurs.
\end{prop}

\begin{proof}It is clear that $-K_X$ is big, since $\lvert -K_X\rvert$ contains the reducible linear system generated by the quartics containing~$S$.\\

Suppose first that $C$ is contained in a plane, which implies that $g=\frac{(d-1)(d-2)}{2}$ and that $r=l-df$, where $l$ is the pull-back of a line of $\p^3$. If $d\ge 2$, the curves numerically proportional to $r$ cover the strict transform of the plane. If $d=4$, these curves intersect trivially $K_X$, so $-K_X$ is nef but the anticanonical morphism contracts the strict transform of the plane. If $d>4$ these curves intersect negatively $-K_X$ so $\lvert -mK_X \rvert$ is not movable and not nef. If $d\le 3$, these intersect positively $-K_X$ so $X$ is Fano. The cases $d=1,2,3$ yields $g=0,0,1$ and the three Sarkisov links are described in Example \ref{ex:plane}.\\

Suppose now that $C$ is contained in a smooth quadric $Q\subset\p^3$, isomorphic to $\p^1\times \p^1$. 
On this surface $C\sim af_1+bf_2$, where $f_1,f_2$ are the two fibres of the two projections, and the restriction of a hyperplane section is $H=f_1+f_2=-\frac{1}{2}K_Q$.
Thus we find $d=H\cdot C=a+b$ and $-2+2g=C^2-2CH=2ab-2a-2b$, so $g=(a-1)(b-1)$. 
We can assume that $a\ge b$ and that $a\ge 2$ (otherwise $C$ is contained in a plane). 

We prove now that $r=l-af$. 
Suppose for contradiction that $r=ml-nf$ with $n>am$ is the strict transform of a curve of degree $m$ passing $n$ times through $C$, its intersections with the quadric $Q$ is $2m\le am<n$, so the curve is contained in the quadric. 
On $Q$ it is equivalent to $\alpha f_1+\beta f_2$ with $\alpha+\beta=m$, and its intersection with $C$ is $b \alpha + a \beta\le a\alpha+a\beta=am<n$, a contradiction.  

If $a\ge 4$,  all curves equivalent to $r$ cover the strict transform of $Q$, and $-K_X$ is not ample. 
Moreover, $-K_X$ is nef if and only if $a=4$, and in this case the anticanonical morphism  contracts the strict transform of $Q$. The possibilities for $a=4$ are $(a,b)=(4,1)$, $(4,2)$, $(4,3)$ and $(4,4)$, corresponding to $(g,d)=(0,5)$, $(3,6)$, $(6,7)$ and $(9,8)$. If $a\ge 5$, $-K_X$ is not nef and thus $-mK_X$ is not movable; since $g=(a-1)(b-1)$ and $d=a+b$, the possibilities where $4d-30\le g\le 14$ are $(0, 6)$, $(0, 7)$, $(4, 7)$, $(5, 8)$, $(6, 9)$, $(8, 8)$, $(10, 9)$, $(12, 9)$, $(12, 10)$, $(14, 11)$.

If  $a\le 3$, $-K_X$ is ample and the contraction of all curves equivalent to $r$ yields a contraction $X\to Y$. 
The possibilities are $(a,b) = (2,1)$, $(2,2)$, $(3,1)$, $(3,2)$ and $(3,3)$, corresponding to $(g,d) = (0,3)$, $(1,4)$, $(0,4)$, $(2,5)$ and $(4,6)$: see Example \ref{ex:quadric}.

Note that when $a=3$ and $b\in \{1,2,3\}$, $C$ is not contained in another quadric because $2H-C$ is not effective on $Q$. \\

Suppose finally that $C$ is contained in a quadric cone $V\subset \p^3$. Blowing-up the singular point yields the Hirzebruch surface $\mathbb{F}_2$, whose Picard group is generated by $f_0$ and $s$, where $f_0$ is a fibre and $s$ the exceptional section of self-intersection $-2$. On $\mathbb{F}_2$, $C$ is equivalent to $as+(2a+e)f_0$ for some $a,e\in \mathbb{N}$. Since $C$ is smooth on $V$, its intersection  $e=C\cdot s$ with $s$ on $\mathbb{F}_2$ is equal to $0$ or $1$; moreover  $a=C\cdot f_0\ge 0$. The restriction of a hyperplane section of $\p^3$ on $\mathbb{F}_2$ is $H=s+2f_0=-\frac{1}{2}K_{\mathbb{F}_2}$. This implies that $d=C\cdot H=2a+e$ and that $-2+2g=C^2-2C\cdot H=ea+a(2a+e)-2(2a+e)$, so  $g=(a-1)(a+e-1)$. 

We can assume that $C-2H$ is not effective, since otherwise $C$ would be contained in a plane or a smooth quadric. This means that either $(a,e)=(2,1)$ or $a\ge 3$. 

We now prove that $r=l-(a+e)f$, corresponding to the transform of a rule on $V$.
Suppose for contradiction that $r=ml-nf$ with $n>(a+e)m$ is the strict transform of a curve $C'$ of degree $m$ passing $n$ times through $C$.
Its intersections with the quadric $V$ is $2m\le (a+e)m <n$, so the curve is contained in the quadric. 
On $\mathbb{F}_2$, the curve $C'$ is equivalent to $\alpha s+m f_0$, with $\alpha=C'\cdot f_0\ge 0$
and its intersection with $C$ is $e\alpha+am$.
We have $C \cdot s = e$, $C'\cdot s = m - 2 \alpha$, so the contraction of $s$ produces $e(m-2\alpha)$ new intersecting points.
Hence the number of intersecting points of $C$ and $C'$ in $\cpt$ is $e\alpha+am + e(m-2\alpha)\le(a+e)m <n$, a contradiction.  

If $a+e\ge 4$,  all curves equivalent to $r$ cover the strict transform of $Q$ and $-K_X$ is not ample. Moreover, $-K_X$ is nef if and only if  $a+e=4$, which corresponds to $(a,e)=(3,1)$ and $(4,0)$, and give respectively $(g,d)=(6,7)$ and $(9,8)$. In this case the anticanonical morphism contracts the strict transform of $Q$. 

If $a+e\ge 5$,  $-K_X$ is not nef and thus $\lvert -mK_X \rvert$ must have a fixed component for all $m$; moreover $g=(a-1)(a+e-1)\ge 16$.
The possibilities where $4d-30\le g\le 14$ are $(0, 6)$, $(0, 7)$, $(4, 7)$, $(5, 8)$, $(6, 9)$, $(8, 8)$, $(10, 9)$, $(12, 9)$, $(12, 10)$, $(14, 11)$.

If $a+e \le 3$, $-K_X$ is ample and the possibilities for $(a,e)$ are $(2,1)$ and $(3,0)$, which give respectively $(g,d) = (2,5)$ and $(4,6)$: see Example \ref{ex:quadriccone}.
\end{proof}
\begin{coro}\label{coro:planeQuadricBN}
Let $C\subset S \subset \p^3$ be a smooth curve of genus $g$ and degree~$d$ which is contained in a plane, a smooth quadric or a quadric cone $S$. Let $\pi\colon X\to \p^3$ be the blow-up of $C$.

Then the following conditions are equivalent:
\begin{enumerate}[$(i)$]
\item
$-K_X$ is ample;
\item
there is no $4$-secant line to $C$;
\item
$(g,d)\in \mathcal{A}_1=\{(0,1),(0,2),(0,3),(0,4),(1,3),(1,4),(2,5),(4,6)\}$.
\end{enumerate}

The following are also equivalent:

\begin{enumerate}[$(i)$]
\item
$-K_X$ is big and nef;
\item
there is no $5$-secant line to $C$;
\item
$(g,d)\in \mathcal{A}_1 \cup \{(0,5),(3,4),(3,6),(6,7),(9,8)\}$.
\end{enumerate}
\end{coro}
\begin{proof}
The proof is the same for a Fano or weak Fano threefold. Implication $(i)\Rightarrow (ii)$ follows from Lemma \ref{lem:badcurve}. Implications $(ii)\Rightarrow (iii)\Rightarrow (i)$ follow from Proposition~\ref{pro:quad}.
\end{proof}

\begin{exple}
\label{ex:plane}
If $C \subset P \subset \p^3$ is a smooth plane curve of degree $d$, we have $g = \frac{(d-1)(d-2)}{2}$. The blow-up of $C$ yields the following Sarkisov links.
\begin{enumerate}[(i)]
\item If $d=1$, the blow-up of the line $C$ produces a link of type I. 
$X$ is a $\cpd$-bundle over $\cpo$, where the fibres correspond to planes through $C$.
\item If $d=2$, the blow-up of the conic $C$, followed by the contraction of the plane $P$ on a smooth point, produces a link of type II. 
$Y$ is a quadric threefold.
\item If $d=3$, the blow-up of the plane cubic $C$, followed by the contraction of the plane $P$ on a terminal point (contraction of type E5), produces a link of type II. 
$Y$ is a singular Fano threefold.
\end{enumerate}
\end{exple}

\begin{exple}
\label{ex:quadric}
If $C \subset Q \subset \p^3$ is a smooth curve of bidegree $(a,b)$ in a smooth quadric surface $Q$, we have $g = (a-1)(b-1)$ and $d = a+b$. 
The blow-up of $C$ yields the following Sarkisov links.
\begin{enumerate}[(i)]
\item If $(a,b) = (1,2)$, $(g,d) = (0,3)$. $C$ is a twisted cubic, and $X$ is a $\cpo$-bundle over $\cpd$, where fibres correspond to bisecants to $C$. Link of type I.
\item If $(a,b) = (2,2)$, $(g,d) = (1,4)$. $C$ is the base locus of a pencil of quadrics, and $X$ is a del Pezzo fibration over $\cpo$. Link of type I.
\item If $(a,b) = (3,1)$, $(g,d) = (0,4)$.  We can contract the transform of $Q$ to a smooth rational curve to obtain $Y$ a smooth prime Fano threefold. Link of type II.

We compute now $-K_Y^3$ to determine the nature of $Y$.
Since the exceptional divisor of $X \to Y$ is $Q = \cpo \times \cpo$, the normal bundle of $\Gamma$ in $Y$ has the form $\No_{\Gamma/Y} = \Ol(a) \oplus \Ol(a)$.
We have $K_Y\cdot \Gamma = -2a-2$, and by Lemma \ref{lem:K^3}
$$(-K_Y)^3 = (-K_X)^3 -2K_Y\cdot \Gamma - 2 = (-K_X)^3  + 4a + 6.$$
Now if $l \subset Q$ is a contracted fibre, and $s \subset S$ is a fibre of the other ruling, we have $Q\cdot l = -1$ and $Q\cdot s = a$. 
Thus
$$K_X\cdot s = \sigma^* K_Y\cdot s + Q\cdot s = K_Y\cdot \Gamma + a = -a-2.$$
Since $C$ was a curve of bidegree $(3,1)$, we have $K_X \cdot s = -3$, hence $a = 1$. 
On the other hand $(-K_X)^3 = 62 - 8.4 = 30$, hence $(-K_Y)^3 = 30 + 4a+ 6 = 40$.
We conclude that $Y$ is an index 2 Fano threefold $V_5 \subset \p^6$.

\item If $(a,b) = (3,2)$, $(g,d) = (2,5)$. We can contract the transform of $S$ to a smooth rational curve to obtain $Y$ a smooth prime Fano threefold. Link of type II.

A similar computation as in the previous example yields $(-K_Y)^3 = 32$.
We conclude that $Y$ is an index 2 Fano threefold $V_4 \subset \p^5$, smooth intersection of two quadrics.

\item If $(a,b) = (3,3)$, $(g,d) = (4,6)$. We can contract the transform of $S$ to a terminal point (contraction of type $E3$). Link of type II.
\end{enumerate}
\end{exple}

\begin{exple}
\label{ex:quadriccone}
If $C \subset V \subset \p^3$ is a smooth curve in a quadric cone $V$, recall that blowing-up the singular point yields the Hirzebruch surface $\mathbb{F}_2$, whose Picard group is generated by $f_0$ and $s$, where $f_0$ is a fibre and $s$ the exceptional section of self-intersection $-2$.
On $\mathbb{F}_2$, $C$ is equivalent to $as+(2a+e)f_0$ for some $a,e\in \mathbb{N}$. 
% Since $C$ is smooth on $V$, its intersection  $e=C\cdot s$ with $s$ on $\mathbb{F}_2$ is equal to $0$ or $1$; moreover  $a=C\cdot f_0\ge 0$. The restriction of a hyperplane section of $\p^3$ on $\mathbb{F}_2$ is $H=s+2f_0=-\frac{1}{2}K_{\mathbb{F}_2}$. This implies that $d=C\cdot H=2a+e$ and that $-2+2g=C^2-2C\cdot H=ea+a(2a+e)-2(2a+e)$, so  $g=(a-1)(a+e-1)$. 
\begin{enumerate}[(i)]
\item If $(a,e ) = (2,1)$, $(g,d) = (2,5)$, we can contract the transform of $V$ to a smooth rational curve to obtain a prime Fano $Y$. Link of type II.

A similar computation as in Example \ref{ex:quadric}(iv) yields $(-K_Y)^3 = 32$.
So again $Y$ is an index 2 Fano threefold $V_4 \subset \p^5$.

\item If $(a,e ) = (3,0)$, $(g,d) = (4,6)$.  We can contract the transform of $V$ to a terminal point (contraction of type $E4$). Link of type II. 
\end{enumerate}
\end{exple}

\section{Curves in cubics}
\label{sec:cubic}

Recall that any smooth cubic surface of $\p^3$ is the blow-up of $6$ points of $\p^2$ such that no $3$ are collinear and no $6$ are on a conic. 
More generally, if $\sigma\colon S\to \p^2$ is the blow-up of $6$ points, maybe some infinitely near, such that $-K_S$ is nef (or equivalently such that all curves of negative self-intersection are smooth rational $(-1)$-curves or $(-2)$-curves), the anticanonical map $\eta\colon S\to \p^3$ is a birational morphism to a normal cubic surface, which contracts all $(-2)$-curves of $S$ (it is an isomorphism if and only if $-K_S$ is ample, which means that $S$ is del Pezzo).

Conversely, all normal cubic surfaces except cones over smooth cubic curves  are obtained in this way.
Indeed such a surface $\hat S$ admits only double points, is rational (project from a singularity) and satisfies $H^1(\hat S,\mathcal O_{\hat S}) = H^2(\hat S,\mathcal O_{\hat S}) = 0$ (use the exact sequence $0 \to \mathcal O(-3) \to \mathcal O \to \mathcal O_{\hat S} \to 0$); hence by \cite[Proposition 8.1.8 (ii)]{Dolgachev} $\hat S$ admits only rational double points. Alternatively one can use the classification of singularities on cubic surfaces in 
\cite[\S 2]{BW}, which does not rely on a cohomological argument. 
Then the minimal resolution $S \to \hat S$ is a weak del Pezzo surface and is the blow-up of $6$ points of $\p^2$ by \cite[Theorem 8.1.13]{Dolgachev}.
For more details the reader can consult \cite[\S 8.1]{Dolgachev} and \cite{Demazure}.

\begin{setup}
\label{setup:cubic}
In this section we consider $C\subset \hat S \subset \p^3$ a smooth curve of genus $g$ and degree $d$ which is contained in a normal rational cubic $\hat S$, but not in a quadric or a plane.

We denote by  $\eta\colon S\to \hat S$ the minimal resolution of singularities of $S$, by $\sigma \colon S\to \p^2$ a birational morphism, which factorises as $\sigma=\sigma_1\circ\dots\circ\sigma_6$, where $\sigma_i$ is the blow-up of a point $p_i$. We denote by $E_i\in\Pic{S}$ the total transform of $\sigma_i^{-1}(p_i)$ on $S$, and by $L\in \Pic{S}$ the pull-back of a general line, so that $\Pic{S}=\z L\oplus \z E_1\oplus \dots\oplus \z E_6$, and that the intersection form is the diagonal $(1,-1,\dots,-1)$. 

Since $C$ is not a line on $\p^3$, we have $C\sim k L-\sum m_i E_i$ on $S$, where $k>0$ is the degree of the image of $C$ on $\p^2$, and $m_1,\dots,m_6\ge 0$ are the multiplicities at the points $p_1,\dots,p_6$. The hyperplane section on $\hat S$ is anti-canonical (the map $S\to \hat S$ is given by $\lvert -K_{S}\rvert$), so $d=C\cdot (-K_{S})$.
Since $C$ is smooth on $\p^3$, it is also smooth on $S$, and we have
$-2+2g=C^2+C\cdot K_{S}=C^2-d$ on $S$. This gives the following equalities:
\begin{align*}
d &= 3k-\sum_{i=1}^6 m_i, \\
g &= \frac{(k-1)(k-2)}{2}-\sum_{i=1}^6 \frac{m_i(m_i-1)}{2}.
\end{align*}

Since we assume that $C\subset \p^3$ is not contained in any quadric, we have that $-2K_{S}-C$ is not effective on $S$ (recall that $-K_S$ is the trace of an hyperplane section, so  $-2K_S-C$ is effective if and only if $C$ is contained in a quadric).
Up to reordering, we can also choose that $m_1\ge m_2\ge m_3\ge m_4\ge m_5\ge m_6$, that $p_1\in \p^2$ and that for $i\ge 2$ either $p_i\in \p^2$ or is infinitely near to $p_{i-1}$. We can moreover assume that $k\ge m_1+m_2+m_3$. 
Indeed, if $k<m_1+m_2+m_3$, the three points $p_1,p_2,p_3$ are not collinear and there exists a quadratic map $\p^2\dasharrow \p^2$ having base-points at $p_1,p_2,p_3$ (this is because there is no $(-3)$-curve on  $\hat{S}$); replacing $\sigma\colon S\to\p^2$ with its composition with the quadratic map replaces $k$ with  $2k-m_1-m_2-m_3$, which is strictly less than $k$.  Reordering the points if needed and continuing this process again if $k<m_1+m_2+m_3$, we end up in a case where $k\ge m_1+m_2+m_3$.
\end{setup}

\begin{prop}
\label{pro:smoothcubic}
Let $C\subset \hat S \subset \p^3$ be a smooth curve of genus $g$ and degree $d$ which is contained in a rational normal cubic, but not in a quadric or a plane.

Let $\pi\colon X\to \p^3$ be the blow-up of $C$. The divisor $-K_X$ is big, it is nef if and only if $C$ does not admit any $5$-secant line. We note $\overline{\NE}(X) = \langle f, r\rangle$, where $f$ is a $\pi$-exceptional curve. Assuming Set-up {\rm \ref{setup:cubic}}, we are in one of the following cases:

\begin{enumerate}[$(i)$]
\item
$-K_X$ is ample;  $r=l-3f$ is generated by $3$-secant lines and we have one of the $4$ following cases:
$$\begin{array}{|c|c||c|c|c|c|}
\hline
g & d& k & (m_1,\cdots, m_6) &  \dim \lvert -K_X\rvert\ge  \\
\hline
1 & 5 &3 & (1,1,1,1,0,0) &  14 \\
3 & 6 &4 & (1,1,1,1,1,1) &  12\\
5 & 7 &6 & (2,2,2,2,2,1) &  10\\
10 & 9& 9 & (3,3,3,3,3,3)  & 7\\
\hline
\end{array}$$ 

\item $-K_X$ is nef but not ample, and the anticanonical morphism is a small contraction; curves proportional to $r=l-4f$ are combination of a finite number of $4$-secant lines of $C$ in $\p^3$, all being in $\hat S$, and we have one of the $9$ following cases \parent{the number of $4$-secants is counted with multiplicity, if $\hat S$ is not smooth, some $4$-secants can count twice or more}:
\label{tab:cubicMovableNef}
$$\begin{array}{|c|c||c|c|c|c|}
\hline
g & d &k & (m_1,\cdots, m_6) &  \text{$\#$ of $4$-secants}&  \dim \lvert -K_X\rvert\ge \\
\hline
0 & 5 &2 & (1,0,0,0,0,0) &  1 &13\\
0 & 6 &2 & (0,0,0,0,0,0)  &  6 &9\\
1 & 6 &3 & (1,1,1,0,0,0) &  3 &10\\
2 & 6 &4 & (2,1,1,1,1,0) &  1& 11\\
3 & 7 &4 & (1,1,1,1,1,0) &  5& 8\\
4 & 7 &5 & (2,2,1,1,1,1) &  2& 9\\
6 & 8 &6 & (2,2,2,2,1,1) &  5& 7 \\
7 & 8 &7 & (3,2,2,2,2,2) &  1& 8\\
9 & 9 &7 & (2,2,2,2,2,2) & 6 & 6\\
\hline
\end{array}$$ 
Moreover, if $\hat{S}$ is smooth, the $4$-secant lines correspond to the conic of $\cpd$ that passes through the last five points, to other conics through $5$ points with similar multiplicities, and in the case  $(g,d) = (6,8)$ to the line through the two points of multiplicity $1$.
\item  $-K_X$ is nef but not ample and the anticanonical morphism is a divisorial contraction; $r=l-4f$,  the curves proportional to $r$ cover the surface $\hat S$ and we have one of the following $3$ cases \parent{same remark as in $(ii)$ for the number of $4$-secants}:
$$\begin{array}{|c|c||c|c|c|c|}
\hline
g & d & k & (m_1,\cdots, m_6) &  \text{$\#$ of $4$-secants} \\
\hline
5 & 8 &6 & (2,2,2,2,2,0) &  10\\
12 & 10 &9 & (3,3,3,3,3,2) & 10\\
19 & 12 &12 & (4,4,4,4,4,4) &  27\\
\hline
\end{array}$$

\item $-K_X$ is not nef, and there is a $5$-secant line to $C$ in $\p^3$.
%Ceux contenus dans des quartiques: $(g,d)\in \{(0, 6)$, $(0, 7)$, $(1, 7)$, $(2, 7)$, $(2, 8)$, $(3, 7)$, $(3, 8)$, $(4, 8)$, $(5, 8)$, $(6, 9)$, $(7, 9)$, $(8, 9)$, $(10, 10)$, $(11, 10)$, $(14, 11)$, $(15, 11)$, $(18, 12)$, $(22, 13)\}$
\end{enumerate}
\end{prop}

\begin{rem}Each case occurs in a smooth cubic: we leave the verification as an exercise to the reader (which can be done along the same lines as in Remark \ref{rem:1511} below).
Note that each case also occur on some singular cubics, depending on the singularity of the cubic, or equivalently on the position of the points blown-up by $\sigma\colon S\to \p^2$.

In case $(ii)$, when we speak about the conic passing through the five last points, we speak about the divisor $2L-E_2-E_3-\dots-E_5$ on $S$, which is in general the strict transform of the conic passing through $p_2,\dots,p_5$; if this divisor corresponds to a reducible curve (which is possible when $\hat S$ is singular), the image of this curve is also a $5$-secant line in $\hat S\subset\p^3$.
\end{rem}

\begin{proof}
We use the notation of Set-up \ref{setup:cubic}.

As in the case of curves contained in a quadric, it is clear that $-K_X$ is big, since $\lvert -K_X\rvert$ gives a birational map (it contains the reducible linear system generated by the quartics containing~$S$).

Suppose that one (effective) divisor $D\in \{E_1,L-E_5-E_6,2L-E_2-E_3-E_4-E_5-E_6\}\subset \Pic{S}$ satisfies $D\cdot C\ge 5$ in $S$. 
If $\hat S$ is smooth, $D$ is a $(-1)$-curve of $S=\hat S$, which is a $5$-secant of $D$, implying that $-K_X$ is not nef. 
If $\hat{S}$ is not smooth, $D$ is either a $(-1)$-curve, or a reducible curve consisting of one $(-1)$-curve connected to a chain of $(-2)$-curves. 
Since $D\cdot (-K_{S})=1$, the image of the reducible curve is a line of $\p^3$, which intersects $D$ into  $5$ points, which is a $5$-secant, implying that $-K_X$ is not nef. This gives case $(iv)$. 

We can thus assume that $D\cdot C\le 4$ for any $D\in \{E_1,L-E_5-E_6,2L-E_2-E_3-E_4-E_5-E_6\}$, which implies that $k-m_5-m_6\le 4$, $2k-m_2-m_3-m_4-m_5-m_6\le 4$ and $m_1\le 4$. In particular, $m_i\le 4$ for $i=1,\dots,6$. This gives finitely many possibilities for $(k,m_1,\dots,m_6)$, all listed in cases $(i)$, $(ii)$, $(iii)$. 

%If a conic of $D\subset S$  satisfies $D\cdot C\ge 8$, the surface $S$ is covered by curves equivalent to $D$ which intersect $-K_X$ non-positively, so the strict transform of $S$ is contained in $|-bK_X|$ for any $b\ge 0$; no multiple of $-K_X$ is thus movable. Choosing $D\in \{L-E_6,2L-E_2-\dots-E_6,3L-2E_1-E_2-\dots-E_6\}$, we can thus assume that $k-m_6< 8$, $2k-m_2-\dots-m_6<8$ and $3k-2m_1-m_2-\dots-m_6<8$. We have moreover $k< 12$ since $L$ corresponds to curves of degree $3$ in $\p^3$. This gives finitely many possibilities for $(k,m_1,\dots,m_6)$, all listed in cases $(iii)$, $(iv)$, $(v)$. We will show below that these cases give in fact $-K_X$ movable.

We now show that these cases give in fact $-K_X$ nef.
In all cases listed, Set-up~\ref{setup:cubic} implies that either $L-E_5-E_6$, $2L-E_2-E_3-E_4-E_5-E_6$ or $E_1$ realises the biggest intersection with $C$ among all $27$  divisors $E_i$, $L-E_i-E_j$, $2L-E_1-E_2-E_3-E_4-E_5-E_6+E_i$. Recall that each of these $27$ divisors $D$ gives a line in $\p^3$, intersecting $C$ into $C\cdot D$ points; the $27$ lines are distinct if $\hat{S}$ is smooth, but not if $\hat{S}$ is singular. An easy check shows that the line with the biggest intersection is equivalent to $l-af$ on $X$, where $a=3$ in case $(i)$ and $a=4$ in cases $(ii)$, $(iii)$.

We now prove that it is the extremal ray $r$. This will imply that $-K_X$ is ample in case $(i)$ and nef but not ample in cases $(ii)$, $(iii)$.
Suppose for contradiction that $r=ml-nf$ with $n>am$ is the strict transform of a curve $\Gamma$ of degree $m$ passing $n$ times through $C$, its intersections with the cubic $\hat S$ is $3m\le am<n$, so the curve is contained in the cubic. 
On $S$ the curve $\Gamma$ is equivalent to a sum of $m$ divisors among the $27$  divisors $E_i$, $L-E_i-E_j$, $2L-E_1-E_2-E_3-E_4-E_5-E_6+E_i$: This is a consequence of Mori's Cone Theorem (see \cite[Theorem 3.7]{KM}), since $-K_X\cdot\Gamma < 0$ and since the extremal rays of the cone of curves of $S$ negative against $-K_S$ correspond exactly to the lines in $\hat{S}$. 
Each of the $27$ divisors intersecting at most $a$ times $C$, we get the contradiction.
%We prove now that cases $(iii)$, $(iv)$ give a movable $-K_X$. Since any reducible quartic obtained by adding a plane to $S$ yields an element of $\lvert -K_X\rvert$, the only possible component is the strict transform of $S$. Computing the projective dimension of $\lvert -K_X\rvert$ with Lemma~\ref{lem:Hypersurfaces} (always possible here because  $4>\frac{-2+2g}{d}$  in each case), we obtain the values of the last column of each table, all being at least equal to $5$. This shows that the strict transform of $S$ is not a common component of $\lvert -K_X\rvert$, since the contrary would give a dimension equal to $3$.
\end{proof}

\begin{rem} \label{rem:1511}
There is one particularly interesting example that falls under case $(iv)$ of Proposition \ref{pro:smoothcubic}, namely the case $(15,11)$.
By Lemma \ref{lem:hypersurfaces}, we know that a curve of type $(15,11)$ is contained in a cubic surface, which must be unique since the degree of the curve is greater than 9.
It is easy to construct such a curve $C$ on a smooth cubic: consider the transform of a curve of degree 10 on $\cpd$ with multiplicities $(4,3,3,3,3,3)$ at the six blown-up points. 
Note that such irreducible curves do exist: the linear system $\lvert C \rvert$ contains the transform of many reducible members, such as 5 conics (4 passing through 4 points, and 1 through 3 points), in particular these reducible members are connected and do not admit a common base point, hence by Bertini's Theorem the general member $C$ is smooth and irreducible. 
Observe that the conic through the five points of multiplicities 3 corresponds to a 5-secant line to $C$: we recover the fact that $-K_X$ is not nef.

More generally, if $C$ is \textit{any} smooth curve of type $(15,11)$, by intersecting the (possibly singular) cubic containing $C$ with one of the irreducible quartics containing $C$ (which does exist by a dimension count, using Lemma \ref{lem:hypersurfaces}), we obtain a residual line $L$ which must be a 5-secant to $C$, since $g(C \cup L) = 19$ (see formulas in \S \ref{sec:linkage}). 
In particular, even for $C$ a general element of $\mathcal{H}^S_{15,11}$, the anticanonical divisor of the blow-up $X$ of $C$ is never nef. 

Note that as a consequence the construction in \cite[Example 7 (1) and Table 1, case 28]{Kal} given as a candidate to be a Sarkisov link from $\cpt$ to $X_{10} \subset \mathbb{P}^7$ in fact never happens. Precisely, the starting assumption \textit{``Let $C \subset \cpt$ be a nonsingular curve that is an intersection of nonsingular quartic
surfaces with $(p_a(C), deg C) = (15, 11)$''} is never fulfilled, since by Bezout's Theorem any quartic containing $C$ must also contains the 5-secant line to $C$. 
\end{rem}

\begin{rem} \label{rem:1210}
We can make a similar discussion about case $(12,10)$. By Lemma \ref{lem:hypersurfaces} such a curve is contained in a cubic surface, which must be unique. By taking the residual curves of the intersection of this cubic with a pencil of general quartics containing $C$, we obtain a pencil of $8$-secant conics to $C$, hence the anticanonical morphism always correspond to a divisorial contraction. 

This contrasts with the case $(5,8)$: we will see that a general curve in $\mathcal{H}^S_{5,8}$ does not lie on a cubic, and yields a weak Fano $X$ with a small anticanonical morphism. 
\end{rem}

\begin{rem} \label{rem:1912}
Note that the case $(19,12)$ always corresponds to the complete intersection of a cubic and a quartic. 
To see that such a curve $C$ is contained in a cubic, we modify slightly the argument of the proof of Lemma \ref{lem:hypersurfaces} as follows.
We have $\deg K = 2g-2 = 36 = \deg D$, so $l(K-D) = 1$ and $\dim \lvert D \rvert = 3d - g + 1 = 18$. Since $N - 1  = 19$, we find a (necessarily unique) cubic surface containing $C$. 

Then by a dimension count, using  Lemma \ref{lem:hypersurfaces}, we see that $C$ is also contained in an irreducible quartic. As a consequence any curve of degree $n$ in the cubic surface containing $C$ is $4n$-secant to $C$.
\end{rem}

\begin{exple} \label{ex:linkFano}
If $\hat S$ is smooth, the four cases in Proposition \ref{pro:smoothcubic}(i) correspond to the following Sarkisov links:
\begin{enumerate}[(i)]
\item Case $(g,d) = (1,5)$ is a link of type II to the quadric $Q \subset \cpq$ (blow-down to a curve $(1,5)$ in $Q$): see \cite[p. 117]{MM} or \cite[Example 3 (L.4)]{pan:2011}.
\item Case $(g,d) = (3,6)$ corresponds to the classical general cubo-cubic transformation of $\cpt$. 
This is the only example of a link of type II from $\cpt$ to $\cpt$ where the isomorphism in codimension 1 is in fact an isomorphism: see \cite{Katz}. 
The transformation contracts a surface of degree 8 which is triple along $C$, and which is swept out by the 3-secant lines to $C$.
\item Case $(g,d) = (5,7)$ corresponds to a link of type I, with a resulting conic bundle whose fibres correspond to 6-secant conics to $C$: see \cite[Remark 2(ii)]{Isko-duke}, \cite[Example 6, (a)]{PanOld} or \cite[Appendix A]{CS}. Note that by this last reference there is no need for a genericity assumption on $C$. 
\item  Case $(g,d) = (10,9)$ corresponds to the complete intersection of two cubic surfaces, and so $X$ is a fibration in del Pezzo surfaces of degree~$3$. 
\end{enumerate}
\end{exple}

\begin{exple} \label{ex:linkcubic}
By Proposition \ref{pro:link}, the nine cases in Proposition \ref{pro:smoothcubic}(ii) correspond to Sarkisov links involving a flop.
When $\hat S$ is smooth, all have already been described in the literature:
\begin{enumerate}[(i)]
\item 
In cases $(g,d) = (1,6), (4,7), (7,8)$ we remark that $d$ plus the number of 4-secant lines is equal to 9. So we have a pencil of cubics with base locus equal to the union of $C$ and the 4-secant lines. After blow-up and flop, we obtain a del Pezzo fibration of degree $6,5,4$ respectively. See \cite[Proposition 6.5(27-28-29)]{JPR2}.
\item 
In case $(g,d) = (2,6)$, after blow-up and flop of the unique 4-secant line we obtain a conic bundle: see \cite[Theorem 7.14(16)]{JPR2}. One can verify, considering the residual curves of a pencil of cubics containing $C$ (hence also the unique 4-secant line), that the fibres correspond to 6-secant conics to $C$: This case is a degenerate version of Example \ref{ex:linkFano}(iii). 
\item 
Case $(g,d) = (0,5)$ corresponds to a link from $\cpt$ to itself with flop of the unique 4-secant line: see \cite[Proposition 2.9]{CM}. This case was known classically, see for instance \cite[first case in Table page 185]{SR}. This is a degenerate version of Example \ref{ex:linkFano}(ii), and again the contracted divisor corresponds to the surface swept out by the 3-secant lines to $C$.
\item 
In case $(g,d) = (0,6)$, after the flop of the six 4-secant lines we can contract the transform of $S$, isomorphic to $\p^2$, on a smooth point, obtaining a Fano threefold $Y_{22}$ of genus 12. See \cite[(2.8.1)]{T}.
\item 
In case $(g,d) = (3,7)$, after the flop of the five 4-secant lines we can contract the transform of $S$, isomorphic to $\mathbb{F}_1$, on a line in a Fano threefold $Y_{16}$ of genus 9. See \cite[(6.1)]{Isk} or \cite[\S 2]{KPZ}.
\item \label{case:68}
In case $(g,d) = (6,8)$, after the flop of of the five 4-secant lines we can contract the transform of $S$, isomorphic to $\p^1\times\p^1$, on a terminal point (contraction of type $E3$).
We have $(-K_X)^3 = 10$ by Lemma \ref{lem:K^3}, hence after the $E3$ contraction we obtain a terminal Fano threefold $Z$ with $(-K_Z)^3 = 12$.  
See \cite[3.2, case 4]{CM}. 
\item 
In case $(g,d) = (9,9)$, after the flop of of the six 4-secant lines we can contract the transform of $S$, isomorphic to $\p^2$, on a terminal point (contraction of type $E5$); the terminal Fano $Z$ that we obtain is such that $(-K_Z)^3=\frac{21}{2}$. 
See \cite[3.3, case 3]{CM}.
\end{enumerate}
\end{exple}

\begin{exple} \label{ex:linkcubicsingular}
All cases of Proposition \ref{pro:link} exist on singular cubics. We describe 
one among many other interesting cases.

In case  $(g,d) = (6,8)$, pick $p_1,\dots,p_5\in \p^2$, $p_6$ infinitely near to $p_5$, such that no $3$ of the $p_i$ are collinear, and no $6$ are on a conic. 
Then $\hat S$ has an unique double point, corresponding to the image of the effective divisor $E_5-E_6\in S$. After the flop of the five $4$-secant lines, we can contract the transform of $\hat S$, isomorphic to a quadric cone, on a terminal point (contraction of type $E4$). This was a case left open in \cite[3.2]{CM}. Note that this case is numerically (but not geometrically) equivalent to Example \ref{ex:linkcubic}(\ref{case:68}); in particular the cube of the anti-canonical divisor is again $12$.

\end{exple}

\begin{prop} \label{pro:antinefcubic}
Let $C\subset \hilb$ be a curve of genus $g$ and degree $d$ contained in a smooth quartic surface. 
Assume that $(g,d)$ is one of the cases listed in Table~{\rm \ref{tab:bignefmovable}}, third column; so in particular $C$ is contained in a cubic surface. 

Then the  blow-up $X$ of $C$ has nef anticanonical divisor if and only if there is no $5$-secant line to $C$. 

In particular, for any such pair $(g,d)$ there exists a non-empty Zariski open set  $U \subset \hilb$ such that the  blow-up $X$ of $C \in U$ is weak Fano.
\end{prop}

\begin{proof}
Suppose that $-K_X$ is not nef, which implies the existence of an irreducible curve $\Gamma$ of degree $n \ge 1$ which intersects at least $4n+1$ times $C$.
Note that $\Gamma$ must be inside all quartics (and \textit{a fortiori} cubics) containing $C$, so $\deg(\Gamma) \le 12 - d$.

Assume first that $\deg(\Gamma) = 12 - d$, which means that $\Gamma \cup C$ is a complete intersection of a cubic and a smooth quartic $S$.
Inside $S$ we have $(C+\Gamma)^2 = 36$ and $C^2 = 2g - 2$, $\Gamma^2 = 2g(\Gamma) - 2$. 
By the formula from \S \ref{sec:linkage}, we have $g - g(\Gamma) = \frac32(d - \deg(\Gamma))$, hence  
\begin{align*}
g(\Gamma) &= 18 + g - 3d; \\
g &= 18 + g(\Gamma) - 3d(\Gamma) 
\end{align*}

Inside the quartic $S$ we have:
\begin{align*}
C\cdot \Gamma &= \frac{36 - (2g-2) - (2g(\Gamma) - 2)}{2} \\
&=20-g-g(\Gamma)\\
&= 3\deg(\Gamma) + 2 - 2g(\Gamma)\\
& \le 3\deg(\Gamma) + 2.
\end{align*}
Since $12 - d = \deg(\Gamma) \ge 2$ in the cases under consideration, we obtain
$$C\cdot \Gamma  \le 4\deg(\Gamma).$$

Suppose now that $n=\deg(\Gamma)<12-d$, which means that $\Gamma \cup C$ is contained in the smooth quartic $S$ but is not a complete intersection of $S$ with another surface (since $(g,d) \neq (9,8)$ or $(6,4)$). We can apply Proposition~$\ref{pro:mori}$, and get $8g(\Gamma \cup C)<(n+d)^2$, where $g(\Gamma \cup C)\ge g(C)+g(\Gamma)+4n$, which yields
$$8g(\Gamma)<(n+d)^2-8g-32n=n^2-2(16-d)n+d^2-8g.$$

We call $P_{g,d}$ the polynomial $n^2-2(16-d)n+d^2-8g\in \mathbb{Z}[n]$. 
We find the following polynomials for the pairs $(g,d)$ under consideration.
$$\begin{array}{|c|c||l|c|}
\hline
g & d & P_{g,d}  \\
\hline
0 & 5 & (n - 2) (n - 20) - 15 \\
0 & 6 & (n - 2) (n - 18) \\
1 & 5 & (n - 1) (n - 21) - 4 \\
1 & 6 & (n - 2) (n - 18)-8 \\
2 & 6 & (n - 2) (n - 18) - 16 \\
3 & 6 & (n - 1) (n - 19) - 7 \\
3 & 7 & (n - 2) (n - 16) - 7 \\
\hline
\end{array}
\qquad
\begin{array}{|c|c||l|c|}
\hline
g & d & P_{g,d}  \\
\hline
4 & 7 & (n - 1) (n - 17) \\
5 & 7 & (n - 1) (n - 17) - 8 \\
6 & 8 & (n - 2) (n - 14) - 12 \\
7 & 8 & (n - 1) (n - 15) -7\\
9 & 9 & (n - 1) (n - 13) - 4 \\
10 & 9 & (n - 1) (n - 13) - 12 \\
12 & 10 & (n - 1) (n - 12) - 8 \\
\hline
\end{array}$$
Since  $8g(\Gamma) < P_{g,d}$ and $n<12-d$, we find that $n<2$.

The last assertion follows from Proposition \ref{pro:smoothcubic}, which implies that the Zariski open set under consideration is non-empty (recall that if $-K_X$ is nef then it is big since $(-K_X)^3 > 0$).
\end{proof}

\section{Curves in quartics} 
\label{sec:quartic} 

\subsection{Curves contained in a  singular rational quartic}

The aim of this section is to produce examples of smooth curves that give rise to a nef anticanonical divisor after blow-up. We shall produce these curves  by considering quartic surfaces with irrational singularities; such surfaces were classified in \cite{IN}, in particular the construction below corresponds to \cite[\S 2.4]{IN}.

\begin{setup}
\label{setup:rationalquartic}
In this section we consider a smooth cubic curve $\Gamma_0\subset\p^2$, given by the equation $F_3(x,y,z)=0$, and we choose a smooth quartic curve of equation $F_4(x,y,z)$ that intersects the curve $\Gamma_0$ into $12$ distinct points $p_1,\dots,p_{12}$ such that no $3$ are collinear, no $6$ on a conic, no $8$ on the same cubic singular at one of the $8$ points, and no $9$ are on the same cubic distinct from $\Gamma_0$. Lemma~\ref{lem:GeneralCubicsQuartics} below shows that this holds for a general quartic.
The rational map $\p^2\dasharrow \p^3$ given by $$(x:y:z)\dasharrow (F_4(x,y,z):xF_3(x,y,z):yF_3(x,y,z):zF_3(x,y,z))$$
induces a birational map $\psi\colon \p^2\dasharrow Q$ where $Q\subset \p^3$ is the singular quartic $$Q=\left\{(w:x:y:z)\in \p^3\ \Big| \ wF_3(x,y,z)=F_4(x,y,z)\right\}.$$

Denoting by $\sigma\colon S\to \p^2$ the blow-up of the $12$ points $p_1,\dots,p_{12}$ and by $\Gamma\subset S$ the strict transform of $\Gamma_0$, the map $\psi$ factorises as $\eta\sigma^{-1}$, where $\eta\colon S\to Q$ is a birational morphism which contracts $\Gamma$ onto the point $q=(1:0:0:0)$. We note $E_1,\dots,E_{12}\subset S$ the exceptional curves of $S$ contracted by $\sigma$ onto $p_1,\dots,p_{12}$ and $L\in \Pic{S}$ the transform of a general line of $\p^2$. The curve $\Gamma\subset S$ is linearly equivalent to $3L-\sum E_i$, and an hyperplane section of $Q$ corresponds to $H=4L-\sum E_i\in \Pic{S}$. 
\end{setup}

\begin{lemm}\label{lem:GeneralCubicsQuartics}
Fixing $\Gamma_0$, a general quartic curve intersects $\Gamma_0$ into $12$ points such that no $3$ are collinear, no $6$ are on a conic, no $8$ are on the same cubic singular at one of the points, and no $9$ are on the same cubic distinct from $\Gamma_0$. The set of $12$-uples of points $p_1,\dots,p_{12}\in \Gamma_0$ obtained has dimension~$11$.
\end{lemm}

\begin{proof}
Denoting by $l\in \Pic{\Gamma_0}$ the restriction of an hyperplane section of $\p^2$ on $\Gamma_0$, we have $\lvert l\rvert \simeq \p^2$, $\lvert 2l\rvert \simeq \p^5$, $\lvert 3l\rvert \simeq \p^8$ and $\lvert 4l\rvert \simeq \p^{11}$. This follows from Riemann-Roch or simply from a direct computation  by restricting the curves of $\p^2$ of degree $1$, $2$, $3$, $4$ to $\Gamma_0$.

A general element of $\lvert 4l\rvert $ yields $12$ distinct points of $\Gamma_0$. 
Any element of $\lvert 4l\rvert $ which yields $3$ collinear points or $9$ points on the same cubic distinct from $\Gamma_0$ corresponds to a sum $D_1+D_3$ where $D_1\in \lvert 4l\rvert $ and $D_3\in \lvert 3l\rvert $: 
This is a consequence of the classical $AF+BG$ Theorem of Noether (see e.g. \cite[Corollary page 122]{Ful}).
The dimension of such elements is thus $10$. Similarly, the elements which give $6$ points on a conic also have dimension $10$. If an element $D\in \lvert 4l\rvert $ is the sum of $12$ points  $p_1,\dots,p_{12}$ so that there exists a cubic $C$ passing through $p_1,\dots,p_8$ and being singular at $p_1$, the restriction of $C$ to $\Gamma_0$ gives $2p_1+\dots+p_8\in \lvert 3l\rvert $, so $D=(2p_1+\dots+p_8)+(p_9+\dots+p_{12}-p_1)$. Since $p_9+\dots+p_{12}-p_1$ has degree $3$, it is linearly equivalent to an effective divisor, and once again $D$ decomposes as $D_1+D_3$ as above.
\end{proof}

Let us point out the following easy fact:
\begin{lemm} \label{lem:easypeasy}
Let $C\subset S$ be an irreducible curve, its image $\eta(C)\subset Q\subset \p^3$ is a smooth curve of $\p^3$ if and only if $C$ is smooth and $C\cdot \Gamma=1$ in $S$.
\end{lemm}
\begin{proof}
Denote by $\hat{\eta}\colon X\to \p^3$ the blow-up of $q=(1:0:0:0)$. It follows from Set-up~\ref{setup:rationalquartic} that  the strict transform of $Q$ on $X$ is isomorphic to the smooth surface $S$, and that $\eta\colon S\to Q$ is the restriction of $\hat{\eta}$.

The curve $C\subset S$ is then the strict transform of $\hat{\eta}(C)\subset \p^3$. Denoting by $E\subset X$ the exceptional divisor, the curve $\eta(C)$ is smooth if and only if $C$ is smooth and $C\cdot E=1$ in $X$. Since $\Gamma$ is the intersection of $E$ and $S$, the intersection $C\cdot E$ on $X$ is equal to the intersection $C\cdot \Gamma$ in $S$.
\end{proof}

\begin{lemm}\label{lem:LinesConicsTwistedCubics}
There are exactly $12$ lines contained in $Q$, all passing through $q$ and corresponding to the image of $E_i$ for some $i=1,\dots,12$.

There are exactly $66$ conics in $Q$, all passing through $q$, corresponding to the strict transforms of the lines of $\p^2$ passing through two of the $p_i$.

There are exactly $5544$ twisted cubics $($smooth cubics of genus $0)$ in $Q$, all passing through $q$, which correspond to the strict transforms of the conics of $\p^2$ passing through $5$ of the $p_i$.
\end{lemm}
\begin{proof}

Each $E_i$ is isomorphic to $\p^1$ on $S$, its intersection with $\Gamma$ and $H$ is $1$, so its image in $Q$ is again isomorphic to $\p^1$, of degree $1$ and passing through $\eta(\Gamma)=(1:0:0:0)=q$.

Let $C\subset S$ be a curve distinct from the $E_i$, isomorphic to $\p^1$ and whose image by $\eta$ is again isomorphic to $\p^1$, of degree $d$. It is linearly equivalent to $mL-\sum a_i E_i$, where $m>0$ and  $a_i\ge 0$ for $i=1,\dots,12$, and its intersection with $E$ and $H$ is respectively  $\eps=3m-\sum a_i\in \{0,1\}$ and $d=\eps +m$.

If $d=1$, we find $(m,\eps)=(1,0)$, which means that $C$ is the strict transform of a line of $\p^2$ passing through three of the $p_i$, impossible.

If $d=2$, we find $(2,0)$ or $(1,1)$ for $(m,\eps)$. The first case is impossible since it would give conics through $6$ of the $p_i$. The second case corresponds to the $66$ lines passing through $2$ of the $p_i$.

If $d=3$, we find $(3,0)$ or $(2,1)$ for $(m,\eps)$. The first case is again impossible, it corresponds to cubics passing through $8$ of the $p_i$ and being singular at one of them. The second case corresponds to the $5544$ conics passing through $5$ of the $p_i$.
\end{proof}

\begin{coro} \label{cor:existence}
Taking $D=kL-\sum m_i E_i$ where $(k,m_1,\dots,m_{12})$ are given by Table~{\rm \ref{tab:singquartic}}, if a curve $C\subset S$ is irreducible, smooth and equivalent to $D$, its image in $\p^3$ is a smooth irreducible curve of type $(g,d)$ with no $5$-secant, no $9$-secant conic and no $13$-secant twisted cubic.

\begin{table}[th]
$$\begin{array}{|c|c||c|c|c|c|}
\hline
g & d & k & (m_1, \dots, m_{12}) &  k-m_{11}-m_{12} & 2k-\sum_{i=8}^{12}m_i\\
\hline
0&7&6& (3,2,2,2,2,2,2,2,0,0,0,0)  & 6& 10\\
1&7&6& (3,2,2,2,2,2,2,1,1,0,0,0)  & 6& 10\\
2&7&6& (2,2,2,2,2,2,2,2,1,0,0,0)  & 6& 9\\
2&8&7& (3,3,2,2,2,2,2,2,2,0,0,0)  & 7& 10\\
3&8&7& (3,3,2,2,2,2,2,2,1,1,0,0)  & 7& 10 \\
4&8&7& (3,2,2,2,2,2,2,2,2,1,0,0)  & 7& 9\\
5&8&7& (2,2,2,2,2,2,2,2,2,2,0,0)  & 7& 8\\
6&9&8& (3,3,3,2,2,2,2,2,2,1,1,0)  & 7& 10\\
7&9&8& (3,3,2,2,2,2,2,2,2,2,1,0)  & 7& 9\\
8&9&8& (3,2,2,2,2,2,2,2,2,2,2,0)  & 6& 8\\
10&10&9& (3,3,3,3,2,2,2,2,2,2,1,1)  & 7& 10\\
11&10&9& (3,3,3,2,2,2,2,2,2,2,2,1)  & 6& 9\\
14&11&10& (3,3,3,3,3,2,2,2,2,2,2,2) & 6& 10\\
\hline
\end{array}$$
\caption{The thirteen cases in Corollary \ref{cor:existence}, which are special instances of the curves listed in Table~$\ref{tab:quarticB&N}$.}
\label{tab:singquartic}
\end{table}

\end{coro}

\begin{proof}In each case, we have $\sum m_i=3k-1$, which means that  $C\cdot \Gamma=1$ on $S$, so the image of $C$ in $\p^3$ is smooth by Lemma \ref{lem:easypeasy}.
The degree $d$ of this curve in $\p^3$ is equal to $C\cdot H=4k-\sum m_i=k+1$. Its genus is  $g = \frac{(k-1)(k-2)}{2}-\sum \frac{m_i(m_i-1)}{2}$, and is given in the table.

If $R\subset \p^3$ is a curve of degree $n\in \{1,2,3\}$ which intersects $C$ in at least $4n+1$ points, its intersection with $Q$ is $4n<4n+1$, so it has to be contained in $Q$. By Lemma~\ref{lem:LinesConicsTwistedCubics}, any line of $Q$  corresponds to one of the $E_i$; since each $a_i$ is at most $4$, there is no $5$-secant. Applying again Lemma~\ref{lem:LinesConicsTwistedCubics}, any conic of $Q$ corresponds to a line of $\p^2$ passing through two points; a quick check shows that $k-m_{11}-m_{12}\le 8$ in each case, so there is no $9$-secant conic. 
The twisted cubics correspond to conics by $5$ points, and once again we can see that $2k-m_8-m_9-m_{10}-m_{11}-m_{12}\le 12$. This finishes the proof.
\end{proof}

\begin{prop}\label{pro:caexiste}
Any of the cases of Corollary $\ref{cor:existence}$ exists.
\end{prop}

\begin{proof}
Each of the case is given by $(k,m_1,\dots,m_{12})$, where $\sum_{i=1}^{12} m_i=3k-1$.

We fix the smooth cubic curve $\Gamma_0\subset \p^2$, and denote by $l\in \Pic{\Gamma_0}$ the restriction of an hyperplane section, which is a divisor of degree $3$ on $\Gamma_0$. 
Recall that by Riemann-Roch's formula on the elliptic curve $\Gamma_0$, if $D$ is a divisor of degree $d$ on $\Gamma_0$ then $|D|$ has dimension $d-1$.
In particular, for any set of twelve points $p_1,\dots,p_{12}$ on the curve, since the divisor $kl-\sum m_ip_i$ has degree $1$ there exists a unique point $q$ on $\Gamma_0$ such that $q+\sum m_ip_i\sim  k l$. We choose $12$ distinct points $p_1,\dots,p_{12}$ on the curve, so that that $p_1+\dots+p_{12}\sim 4l$, no $3$ of the $p_i$ are collinear, no $6$ are on the same conic, no $8$ are on the same cubic singular at one of them and no $9$ are on the same cubic distinct from $\Gamma_0$ (possible by Lemma~\ref{lem:GeneralCubicsQuartics}); we can moreover assume that  the point $q\in \Gamma_0$ equivalent to $k l-\sum m_ip_i$ is distinct from the $p_i$. This gives a particular case of Set-up~\ref{setup:rationalquartic}.

In the blow-up $S$ of the twelve points, the divisor $D=kL-\sum m_iE_i$ satisfies $-3k+\sum m_i=D\cdot K_S=-1$. Moreover, the integer $g$ given in the table is equal to $\frac{(k-1)(k-2)}{2}-\sum \frac{m_i(m_i-1)}{2}$ (and is the genus of an irreducible member of $\lvert D\rvert$ if such member exists, which is exactly what we want to prove). The system $\lvert D\rvert$ corresponds to curves of $\p^2$ of degree $k$ having multiplicity $m_i$ at each point $p_i$. Its (projective) dimension is thus at least equal to
\begin{equation} \label{eq:dim}
\frac{(k+1)(k+2)}{2}-\sum \frac{m_i(m_i+1)}{2}-1=g+3k-\sum m_i-1=g.
\end{equation}

Recall that we want to prove that $\lvert D \rvert$ contains a smooth irreducible member, for a general choice of the $p_i$ as above.

The first case in the list  (first line of Table~$\ref{tab:singquartic}$) is special. The divisor $D$ corresponds to a $(-1)$-curve on the blow-up of the points $p_1,\dots,p_{8}$, which is a del Pezzo surface (because of the assumption on the points). There is thus a unique curve equivalent to $D$, which is irreducible, and smooth.

The second case of the list  (second line of Table~$\ref{tab:singquartic}$, where $g=1$) can also be viewed directly. Denote by $\mu\colon S\to S_2$ the contraction of $E_9,\dots,E_{12}$ and of the $(-1)$-curve equivalent to $3L-2E_1-E_2-\dots-E_7$. The surface $S_2$ is a del Pezzo surface of degree $2$ since the blow-up of $p_1,\dots,p_8\in \p^2$ is a del Pezzo surface. Moreover, $D=\mu^{*}( -K_{S_2})-E_9$. Taking a general member of the subsystem of $\lvert -K_{S_2}\vert$ of elements passing through $\mu(E_9)$, the strict transform will yield a smooth elliptic curve equivalent to $D$ (note that the member will not pass through the points $\mu(E_{10})$, $\mu(E_{11})$, $\mu(3L-2E_1-E_2-\dots-E_7)$).

It remains to study the cases where $g\ge 2$, which implies  $\dim \lvert D\rvert \ge 2$. 
Recall that $\Gamma \subset S$, the strict transform of $\Gamma_0$, has self-intersection $-3$ and is equivalent to $-K_S$. 
Any effective divisor $C$ equivalent to $D$ either contains $\Gamma$ or intersect it into $C\cdot \Gamma=D\cdot (-K_S)=1$ point. In particular, the restriction of the linear system $D$ to $\Gamma$ is equivalent to $kl-\sum m_i p_i\sim q$, so $C\cap \Gamma$ is either $\Gamma$ or $\{q\}$.
The system $\lvert D\rvert$ has (at least) one base-point, which is the point $q\in \Gamma$.

We need the following fact.

\begin{fact} \label{fact}
The system $$l_D=\left\lvert D+K_S-\sum\limits_{m_i=0} E_i\right\rvert$$ has no base-point outside of $\Gamma \cup \bigcup\limits_{m_i=0} E_i$.
\end{fact}

The proof of this fact proceeds by a case by case analysis.

In case $(2,7)$ (third line of Table~$\ref{tab:singquartic}$), $l_D=\lvert D-\Gamma-E_{10}-E_{11}-E_{12}\rvert=\lvert 3L-E_1-\dots-E_8\rvert$, which has dimension $\ge 1$ by formula (\ref{eq:dim}).
%, square $1$ and intersection $1$ with the anti-canonical divisor. 
Remark that $\pi_{*}(l_D)$, the image of the system  on $\p^2$, corresponds to the linear system of cubics passing through the points $p_1, \dots, p_8$.
On the elliptic curve $\Gamma_0$, the linear system  $\lvert 3l-\sum_{i=1}^8 p_i \rvert$ contains exactly one element (by Riemann Roch's formula), and thus this point $q'$ of $\Gamma_0$ is a base point of $\pi_{*}(l_D)$. Because of Set-up~\ref{setup:cubic}, the point $q'$ is distinct from the $p_i$ (see Lemma~\ref{lem:GeneralCubicsQuartics}).
Observe that the system $\pi_*(l_D)$ has no reducible member: this would implies that there exists a line through 3 points or a conic through 6 points. 
Thus the general member of $l_D$ is the strict transform of an irreducible cubic through the points $q', p_1,\dots,p_8$; we conclude that the lift of $q'$ on $S$, which lies on $\Gamma$, is the unique base-point of the system $l_D$.
% and has dimension $1$.

In case $(2,8)$, a similar argument holds, applied to the mobile system $l_D=\lvert 4L-2E_1-2E_2-E_3-\dots-E_9\rvert$, which has also dimension at least $1$.  For a general choice of the $p_i$, we can assume that the point $q'$ of $\Gamma_0$ corresponding to $\lvert 4l-2p_1-2p_2-\sum_{i=3}^9 p_i \rvert$ is distinct from the $p_i$ (for each $i$, the case $q'=p_i$ gives a strict closed subset of the $11$-dimensional space, as in Lemma~\ref{lem:GeneralCubicsQuartics}).
The only reducible members of $\pi_{*}(l_D)$ are union of a line passing through $p_1,p_2$ and a cubic by the points $p_1,\dots,p_9$, this latter being only $\Gamma_0$ by Set-up~\ref{setup:cubic}. The dimension of $l_D$ being at least $1$,  a general member comes from an irreducible quartic of $\p^2$ passing through $p_1,\dots,p_9$, with multiplicity $2$ at $p_1,p_2$. Since $(l_D)^2=1$, the lift on $S$ of the point $q'$, which lies on $\Gamma$, is the unique base-point of $\lvert l_D\rvert$.

In cases $(3,8)$, $l_D=\lvert 4L-2E_1-2E_2-E_3-\dots-E_8\rvert$ has no base-point outside of $\Gamma\cup E_9$, because it contains the union of $E_9$ with the system of case $(2,8)$.

For $(4,8)$, $l_D=\lvert 4L-2E_1-E_2-E_3-\dots-E_8\rvert$ and the result follows from the fact that this system contains the union of $\lvert L-E_1\rvert$, which has no base-point, with the system of case $(2,7)$.

In case $(5,8)$, $l_D=\lvert 4L-\sum_{i=1}^{10} E_i\rvert$ has no base-point outside of $E_{11}\cup E_{12}$, since it contains the union of $E_{11}+E_{12}$ with the base-point free system $\lvert 4L-\sum_{i=1}^{12} E_i\rvert$.
Indeed the latter system corresponds to the morphism  $S\to\p^3$, which is an isomorphism outside of $\Gamma_0$. 

The cases  $(6,9)$, $(7,9)$ and $(8,9)$ give  $\lvert 5L-2E_1-2E_2-2E_3-\sum_{i=4}^{9}E_i\rvert $, $\lvert 5L-2E_1-2E_2-\sum_{i=3}^{10}E_i\rvert$ and $\lvert 5L-2E_1-\sum_{i=2}^{11}E_i\rvert$ for $l_D$. 
The system contains the union of the unique $(-1)$-curve $C_1\subset S$ equivalent to  
$3L-2E_1-\sum_{i=2}^{7} E_i$ with a base-point free system corresponding to conics of $\p^2$ by four of the $p_i$. 
Similarly it also contains the union of the curve $C_2$ equivalent to $3L-2E_1-\sum_{i=3}^{8} E_i$ with another base-point free system. 
Since $C_1\cdot C_2=0$, the system $l_D$ is base-point free.

The cases $(10,10)$, $(11,10)$ and $(14,11)$ yield $\lvert 6L-2\sum_{i=1}^4 E_i-\sum_{i=5}^{10} E_i \rvert$,  $\lvert 6L-2\sum_{i=1}^3 E_i-\sum_{i=4}^{11}E_i \rvert$ and $\lvert 7L-2\sum_{i=1}^5 E_i-2E_5-\sum_{i=6}^{12}E_i \rvert$ for $l_D$. For $i\not=j$, we denote by $R_{ij}\subset S$ the $(-1)$-curve equivalent to $L-E_i-E_j$. For $(i,j)=\{(8,9),(8,10),(9,10)\}$, the system $l_D$ contains the union of $C_1\cup R_{ij}$ with a base-point free system. This implies that $l_D$ has no base-point except maybe on $C_1$. Replacing $C_1$ by $C_2$, we see that $l_D$ has no base-point. This ends the proof of the Fact \ref{fact}.\\

Now that we have proved that $l_D=\lvert D+K_S-\sum_{m_i=0} E_i\rvert$ has no base-point outside of $\Gamma\cup \bigcup_{m_i=0} E_i$, this implies that the same property holds for $\lvert D\rvert$. Indeed, taking a member $C\in \lvert l_D\rvert$, the reducible effective divisor $C+\Gamma+\sum_{m_i=0} E_i$ is a member of $\lvert D\rvert$.

This implies  that $D$ is nef.  Indeed, if an irreducible curve intersects negatively $D$, it is a fixed component of $\lvert D\rvert$ and has to be $\Gamma$ or $E_i$ for some $i$. Each of these curves intersects non-negatively $D$, so $D$ is nef. 
The fact that $D^2>0$ implies that $D$ is also big, and this gives $h^i(D+K_S)=0$ for $i=1$ or $2$.
Thus $\dim \lvert D+K_S\rvert= (D+K_S)\cdot D / 2 +1\le \dim \lvert D\rvert-1$. 
A general member $C$ of $\lvert D\rvert$ does not contain $\Gamma\sim -K_S$ and thus $C\cap \Gamma=\{q\}$, so $q$ is the only base-point of $\lvert D\rvert$ on $\Gamma$. 
Moreover the image of $C$ on $\p^2$ is a curve of degree $k$ intersecting $\Gamma_0$ into $3k$ points (counted with multiplicity), which have to be $\sum m_ip_i+q$. If $m_i=0$, the point $p_i\in \p^2$ is therefore not a point of $\pi(C)$, which means that $C\cdot E_i=0$. This shows that $q$ is the only base-point of $\lvert D\rvert$ on $S$. 

By Bertini's theorem $C$ is smooth outside of $q$, and since $C\cdot \Gamma = 1$ we also have $C$ smooth at $q$. 
\end{proof}

\subsection{Curves in a smooth quartic}

\begin{prop} \label{pro:antinef}
Let $C\in \hilb$ be a curve of genus $g$ and degree $d$. 
Assume that $C$ is contained in a smooth quartic surface, and that $(g,d)$ is one of the thirteen cases listed in Table~{\rm \ref{tab:quarticB&N}}. 

Then the  blow-up $X$ of $C$ has nef anticanonical divisor if and only if there is no $(4n+1)$-secant rational smooth curve of degree $n$ for $n=1,2,3$. 

The $13$-secant twisted cubics are possible only if $(g,d)\in \{(0,7),(2,8),(3,8)\}$. 

The $9$-secant conics are possible only if $(g,d)\in \{(6,9),(7,9)\}$.

The $5$-secant lines are possible in all cases except  for $(g,d)=(14,11)$.
\end{prop}

\begin{proof}
Suppose that $-K_X$ is not nef, which implies the existence of an irreducible curve $\Gamma$ of degree $n$ which intersects at least $4n+1$ times $C$. We want to prove that this curve is smooth, rational and has degree $\le 3$, and describe when it could exist. Note that $C$ must be inside all quartics containing $C$, so $\deg(\Gamma) \le 16 - d$.

Assume first that $\deg(\Gamma) = 16 - d$, which means that $\Gamma \cup C$ is a complete intersection of two quartics, which by assumption can be taken smooth.
Inside one of the quartics we have $(C+\Gamma)^2 = 64$ and $C^2 = 2g - 2$, $\Gamma^2 = 2g(\Gamma) - 2$. 
By the formula from \S \ref{sec:linkage}, we have $g - g(\Gamma) = 2(d - \deg(\Gamma))$, hence $g(\Gamma) = g - 4d + 32 \ge 2$ in all cases under consideration. The formula being symmetric, we find $$g = 32 + g(\Gamma) - 4\deg(\Gamma).$$  

Hence (inside a quartic)
\begin{align*}
C\cdot \Gamma &= \frac{64 - (2g-2) - (2g(\Gamma) - 2)}{2} \\
&=34-g-g(\Gamma)\\
&= 4\deg(\Gamma) + 2 - 2g(\Gamma)\\
& < 4\deg(\Gamma). 
\end{align*}
For the last line we used $g(\Gamma) \ge 2$.\\

Now assume that $C \cup \Gamma$ is the complete intersection of a smooth quartic with a cubic. 
As in the first part of the proof of Proposition \ref{pro:antinefcubic} we can prove that
$$ C\cdot \Gamma \le 3\deg(\Gamma) + 2. $$
Note that we have $n \ge 12 - d \ge 2$, since in the case $(14,11)$ the residual curve would have genus $-1$ and degree $1$, which is impossible.
Hence we have $C\cdot \Gamma \le 4\deg(\Gamma). $\\

Suppose finally that $n=\deg(\Gamma)<16-d$, and that $\Gamma \cup C$ is not a complete intersection of $S$ with another surface. 
We can apply Proposition~$\ref{pro:mori}$, and get $8g(\Gamma \cup C)<(n+d)^2$, where $g(\Gamma \cup C)\ge g(C)+g(\Gamma)+4n$, which yields
$$8g(\Gamma)<(n+d)^2-8g-32n=n^2-2(16-d)n+d^2-8g.$$

We call $P_{g,d}$ the polynomial $n^2-2(16-d)n+d^2-8g\in \mathbb{Z}[n]$. 
For $(g,d)=(14,11)$ we find $8g(\Gamma) < P_{14,11}=n^2 -10 n+9 = (n-1)(n-9)$,
since $1 \le n \le 5$, we have a contradiction. More generally, we find the following polynomials for the other pairs $(g,d)$.
$$\begin{array}{|c|c||l|c|}
\hline
g & d & P_{g,d}  \\
\hline
0 & 7 & (n - 4) (n - 14) - 7 \\
1 & 7 & (n - 3) (n - 15) - 4 \\
2 & 7 & (n - 3) (n - 15) - 12 \\
2 & 8 & (n - 4) (n - 12) \\
3 & 8 & (n - 4) (n - 12) - 8 \\
4 & 8 & (n - 3) (n - 13) - 7 \\
\hline
\end{array}
\qquad
\begin{array}{|c|c||l|c|}
\hline
g & d & P_{g,d}  \\
\hline
5 & 8 & (n - 2) (n - 14) - 4 \\
6 & 9 & (n - 3) (n - 11) \\
7 & 9 & (n - 3) (n - 11) - 8 \\
8 & 9 & (n - 2) (n - 12) - 7 \\
10 & 10 & (n - 2) (n - 10) \\
11 & 10 & (n - 2) (n - 10) - 8 \\
\hline
\end{array}$$
Since  $8g(\Gamma) < P_{g,d}$ and $n<16-d$, we find that $n<4$.

If $n=3$, we have $(g,d)\in \{(0,7),(2,8),(3,8)\}$; the curve $\Gamma$ has to be smooth and rational otherwise it would be contained in a plane and could not intersect $C$ at $13>d$ points.

If $n=2$, the curve is contained in a plane and intersects $C$ at $9$ points, so $d\ge 9$. The only cases where $P_{g,d}(n)\ge 1$ are $(g,d)\in \{(6,9),(7,9)\}$.

The case $n=1$ is always possible, except for $(g,d)=(14,11)$, treated above.
\end{proof}

\subsection{Proof of Theorem \ref{thm:gendelpezzo}}
\label{sec:proof}

We consider the blow-up $X \to \cpt$ of a curve $C \in \hilb$ of genus $g$ and degree $d$.\\

We first assume that $-K_X$  is big and nef. 
Then $C$ does not admit any $5$-secant line, 
$9$-secant conic or $13$-secant twisted cubic (Lemma \ref{lem:badcurve}).
%Moreover, by Proposition \ref{pro:insmoothquartic}, $C$ is contained in a pencil of quartics whose general member is smooth, in particular $C$ is admissible and $d \le 16$.

If $C$ is contained in a quadric or a plane,  the exhaustive study in Section~\ref{sec:quad}, and in particular Corollary~\ref{coro:planeQuadricBN}, implies that $(g,d)$ is one of the cases in the first two columns of Table \ref{tab:bignefmovable} which are not crossed.

We assume now that $C$ is not contained in any quadric.
By Riemann-Roch on a threefold and Kawamata-Viehweg vanishing, we have $\dim \lvert -K_X\rvert = \frac{1}{2}(-K_X^3) + 2 \ge 3$. Thus $C$ is contained in a pencil of quartics, whose general member can be taken smooth by Proposition \ref{pro:insmoothquartic}.
In particular $d \le 16$.

If $C$ is the complete intersection of a smooth quartic with another surface $F$, then $F$ has to be a cubic:  This yields the case $(g,d) = (19,12)$.

On the other hand if $C$ is not a complete intersection, by Proposition \ref{pro:mori} we have
$$8g < d^2.$$
Since $(-K_X)^3 > 0$, by Lemma \ref{lem:K^3} we have 
$$4d -30 \le g.$$
Putting together both inequalities yields 
$d^2 > 8(4d -30)$, hence
$$ d^2 - 32d + 240 = (d-12)(d-20) > 0.$$
Since we already know that $d \le 16$, we obtain $d \le 11$, hence $g\le 15$.
But $g = 15$ implies $d = 11$, and by Remark~\ref {rem:1511} curves of type $(15,11)$ always admit a $5$-secant line; so we obtain $g\le 14$. 

Since $C$ is not in any quadric we have the Castelnuovo bound on the genus of $C$ (see \cite[Theorem IV.6.4]{Har}):
$$g < \lfloor \tfrac{d^2}{4}\rfloor-d+1.$$
Pairs $(g,d)$ which satisfy this inequality together with $8g < d^2$ and $4d -30 \le g \le 14$ are in the last two columns of Table~\ref{tab:bignefmovable}, or equal to $(0,4)$.
But curves of type $(0,4)$ are contained in quadrics (Lemma~\ref{lem:hypersurfaces}).\\

We now assume that $-K_X$  is ample. It is clear by Lemma \ref{lem:badcurve} that $C$ admits no $4$-secant line. 
Then we use a classical formula of Cayley, in the following precise formulation as proved by Le Barz \cite{Barz}.

\begin{prop}
Let $C \in \hilb$. Assume that $C$ does not admit infinitely many $n$-secant lines for any  $n \ge 4$. 
Then the number of $4$-secant lines to $C$, counted with multiplicity, is given by the formula
$$ \frac{(d-2)(d-3)^2(d-4)}{12} - \frac{(d^2 - 7d +13 -g)g}{2}.$$ 
\end{prop}

In particular, among the $(g,d)$ that could produce a weak Fano threefold, this formula gives $0$ precisely in the cases listed in the subsets $\A_1$ and $\A_2$ in Theorem \ref{thm:gendelpezzo}.\\

Conversely, assume that $C \in \hilb$ with $(g,d) \in \A_1$, or $(g,d) \in \A_2$ and $C$ does not admit a $4$-secant line. 
We want to prove that $X$ is Fano.
If $C$ lies on a quadric, the result follows from Corollary \ref{coro:planeQuadricBN}.  
If $C$ does not lie on a quadric, then $(g,d)$ is in the set 
$$\A_2 = \{(1,5), (3,6), (5,7), (10,9)\}.$$
In these cases Lemma \ref{lem:hypersurfaces} tells us that $C$ is contained in a pencil of (irreducible) cubic surfaces. 
If $\Gamma$ is an irreducible curve of degree $n$ which is $4n$-secant to $C$, then $\Gamma$ must be in the base locus of this pencil, and $2 \le n \le 9 - d$. 
This rules out $(10,9)$. 
Since the three other cases have degree $d \le 7$, clearly they do not admit a $8$-secant conic: the plane containing the conic has to meet $C$ in $d$ points only.
By a similar argument $(1,5)$ does not admit a $12$-secant curve $\Gamma$ of degree 3, by considering the plane or the quadric containing $\Gamma$.
Finally, if $\Gamma \cup C$ is a complete intersection with $C$ of type $(1,5)$ (resp. $(3,6)$), by the formulas in \S \ref{sec:linkage} we see that $\Gamma$ is a $10$-secant (resp. $8$-secant) rational curve of degree $4$ (resp. $3$).
We conclude that $X$ is Fano.\\

Finally assume that $C \in \hilb$ is contained in a smooth quartic, with $4d-30\le g\le 14$ or $(g,d) = (19,12)$, and that there is no $5$-secant line, $9$-secant conic, nor $13$-secant twisted cubic to $C$.
If $(g,d)$ is in $\A_1$ or $\A_2$  we  already know that $X$ is Fano.

% If $(g,d) = (19,12)$, then $C$ is the complete intersection of a cubic and a quartic: see Remark \ref{rem:1912}. CAS FACILE: il n'y a jamais de 5-secante.

So we can assume that $(g,d)$ is in the third or fourth column of Table \ref{tab:bignefmovable}.

In the former case, we conclude by Proposition \ref{pro:antinefcubic}, and in the latter one, by Proposition \ref{pro:antinef}, that $X$ is weak Fano. \\

%
%the star next to $(12,10)$ comes from Remark \ref{rem:1210}.
%
%Note that if a curve is of type $(g,d)$ with $(g,d)$ in the second column then it must lie in a quadric: this follows from Proposition \ref{lem:Hypersurfaces}, except in case $(9,8)$ where it follows from the Castelnuovo bound stated below (a similar remark applies for the first column).
%
%The four crossed out cases correspond to Proposition \ref{Prp:Quad}(iii).
%

The fact that curves of type $(3,4), (6,7), (9,8), (12,10)$ and $(19,12)$ yield a divisorial anticanonical morphism was observed in Proposition \ref{pro:quad}$(ii)$ and  Remarks \ref{rem:1210}, \ref{rem:1912}. \\

It remains to prove that Conditions $(i)$ and $(ii)$ in Theorem \ref{thm:gendelpezzo} correspond to non-empty open sets in relevant $\hilb$.

For $(g,d)$ in $\A_2 \cup \A_3$, this follows from the explicit construction of examples on smooth cubic surfaces, as described in Proposition \ref{pro:smoothcubic}.

The case $(g,d) \in \A_4$ is more delicate. 
By Proposition \ref{pro:nasu} we know that $\hilb$ is irreducible in all cases, except for  $(14,11)$ where there are two components. 
We denote by $V_{14,11}$ the irreducible component whose general members lie on a quartic surface, and $V_{g,d} = \hilb$ in the other cases.

By Lemma \ref{lem:hypersurfaces} we know that the system of quartics containing any $C \in V_{g,d}$ has dimension at least 3. 
Consider $W$ the set of pairs $(C,Q)$ where $C \in V_{g,d}$, $Q \subset \cpt$ is a quartic surface and $C \subset Q$.    
By Proposition \ref{pro:mori1} we know that there exists a smooth quartic surface containing a smooth curve of genus $g$ and degree $d$. 
Hence the open set $W_s \subset W$ of pairs $(C,Q)$ such that $Q$ is smooth is non-empty.  
Projecting on the first factor we obtain $V_1 \subset V_{g,d}$ a non-empty open subset such that any $C \in V_1$ is contained in a smooth quartic surface.

From Proposition \ref{pro:caexiste} we know that there exists a non-empty open set $V_2 \subset V_{g,d}$  such that if $C \in V_2$, then
$C$ does not admit any $5$-secant line, $9$-secant conic or $13$-secant twisted cubic.

By irreducibility of $V_{g,d}$, we have $V_1 \cap V_2 \neq \emptyset$: this is the desired open set.\hfill$\square$

\begin{rem} \label{rem:admi}
The assumption that $C$ is contained in a smooth quartic in Theorem \ref{thm:gendelpezzo} is probably superfluous.
The point to remove it would be to prove that for $(g,d) \in \A_3 \cup \A_4$, if $C \in \hilb$ does not have a 5-secant line then $C$ is contained in a smooth quartic.

For instance consider $C$ lying on a smooth cubic surface $S$, coming from the transform of a curve of degree $9$ in $\cpd$ with multiplicities $(3,3,3,3,2,2)$ at the six blown-up points. 
Then $C \in \mathcal{H}^S_{14,11}$, and any quartic containing $S$ breaks into two components, $S$ and a plane. 
Note that $C$ admits $5$-secant lines, for instance the transform of the line through the two points with multiplicity 2. 
More generally, any $C \in \mathcal{H}^S_{14,11}$ which lies on a cubic cannot lie on a irreducible quartic, otherwise $C$ would be linked to a curve with genus $-1$ and degree $1$: contradiction.
But we have been unable to prove that if $C \in \mathcal{H}^S_{14,11}$ lies in a irreducible quartic, then the quartic can be taken to be smooth.
\end{rem}

\subsection{Existence of Sarkisov links} \label{sec:existence}

\begin{prop} \label{pro:sarkisov}
If $(g,d)$ is one of the thirteen cases listed in Table~{\rm \ref{tab:quarticB&N}}, and $C$ is a general curve in $\hilb$, then the blow-up $X$ of $C$ is weak Fano, the anticanonical morphism is a small birational map and thus yields a Sarkisov link involving a flop \parent{see Example~{\rm \ref{rem:detailslinks}} for details}.
\end{prop}

\begin{proof}
By Theorem \ref{thm:gendelpezzo} there exists a non-empty open subset $V \subset \hilb$ such that for any $C \in  V$, the blow-up $X$ of $C$ is weak Fano. 

Now for such an $X$ obtained from $C \in  V$, suppose that the anticanonical morphism is a divisorial contraction.
Then we would be in one of the 24 cases of \cite[Theorem 4.9]{JPR}.
The only coincidence between the two lists are $(5,8)$ and $(11,10)$. 
In case $(5,8)$ the exceptional divisor of the anticanonical map is the transform of a cubic surface through $C$. 
But curves of type $(5,8)$ contained in a cubic surface form a codimension 1 closed subset in the 32-dimensional $\mathcal{H}^S_{5,8}$ (see \cite{GPI}).
The case $(11,10)$ is more subtle, since even in the divisorial case the curve has no reason to lie on a cubic. Nevertheless curves that correspond to a small  anticanonical morphism form again a dense open set in $V_1 \cap V_2$: this follows from the construction given in \cite[No 19 page 617]{JPR}, which involves taking the minors of  a $(4 \times 4)$ matrix $M$ with linear entries in $\C[x_0, \dots, x_4]$. It is noted there that the anticanonical morphism is a divisorial contraction when $M$ is symmetric, and is a small contraction otherwise.

In conclusion, we obtain in all cases a non-empty open subset of $\hilb$ with the desired properties.
The fact that it yields a Sarkisov link was discussed in Section \ref{sec:sarkisov}, and the presence of a flop comes from the fact that $X$ is never Fano in the cases under consideration. 
\end{proof}

\begin{exple} \label{rem:detailslinks}
The numerical possibilities listed in Table \ref{tab:quarticB&N} were already found in the literature, but without proof of actual existence (except maybe for $(g,d) = (11,10)$, see end of the proof of Proposition \ref{pro:sarkisov}, and $(g,d) = (7,9)$, see below).

Note that our proof gives the existence of these links, without knowing \textit{a priori} the end result.

\begin{enumerate}[$(i)$]
\item Cases $(g,d) = (0,7)$, $(2,8)$, $(3,8)$, $(5,8)$, $(6,9)$, $(10,10)$, $(11,10)$ and  $(14,11)$ must correspond to Sarkisov links from $\cpt$ to $\cpt$: see \cite{CM}, cases 90, 49, 75, 99, 50, 51, 76, 52 in their Tables 1-3. 
We plan to come back to the description of these interesting Cremona transformations in a future paper.
\item Case $(g,d) = (1,7)$ corresponds to a Sarkisov link to the prime Fano $X_{22}$ of genus 12, with a flop and contraction to another curve of genus 1 and degree 7: see \cite[case 98]{CM}.
\item Case $(g,d) = (2,7)$ corresponds to a Sarkisov link to the intersection of two quadrics $V_{4} \subset \p^5$, with a flop and contraction to a curve of genus 0 and degree 5: see \cite[case 103]{CM}.
\item Case $(g,d) = (4,8)$ corresponds to a Sarkisov link to the Fano threefold $V_{5} \subset \p^6$, with a flop and contraction to a curve of genus 4 and degree 10: see \cite[case 89]{CM}.
\item Case $(g,d) = (7,9)$ corresponds to a Sarkisov link to the Fano threefold $X_{12}$ of genus 7, with a flop and contraction to a curve of genus 0 and degree 3. The existence of this link is claimed in \cite[p. 103]{IP}, at least in the case of a curve lying in a special smooth quartic surface as constructed in the proof of Proposition \ref{pro:mori1}, but the details of the argument are not given.
\item Case $(g,d) = (8,9)$ corresponds to a Sarkisov link of type I to a fibration in del Pezzo surfaces of degree 5 (after one flop): see \cite[Proposition 6.5(25)]{JPR2}.
\end{enumerate}
\end{exple}

\bibliographystyle{alpha}
\bibliography{biblio}

\begin{thebibliography}{ACM11}

\bibitem[ACM11]{ACM}
Maxim Arap, Joseph~W. Cutrone, and Nicholas~A. Marshburn.
\newblock On the existence of certain weak fano threefolds of picard number
  two.
\newblock {\em preprint arXiv:1112.2611}, 2011.

\bibitem[BW79]{BW}
J.~W. Bruce and C.~T.~C. Wall.
\newblock On the classification of cubic surfaces.
\newblock {\em J. London Math. Soc. (2)}, 19(2):245--256, 1979.

\bibitem[CM10]{CM}
Joseph~W. Cutrone and Nicholas~A. Marshburn.
\newblock Towards the classification of weak {F}ano threefolds with $\rho= 2$.
\newblock {\em preprint arXiv:1009.5036}, 2010.

\bibitem[CS11]{CS}
Ivan Cheltsov and Constantin Shramov.
\newblock Weakly-exceptional singularities in higher dimensions.
\newblock {\em preprint arXiv:1111.1920v2}, 2011.

\bibitem[Dem80]{Demazure}
M.~Demazure.
\newblock {\em {S}urfaces de del {P}ezzo, {II, III, IV, V}}, volume 777 of {\em
  Lecture Notes in Mathematics}, pages 21--69.
\newblock Springer, Berlin, 1980.

\bibitem[Dol11]{Dol-Cremona}
I.~Dolgachev.
\newblock {\em Lectures on {C}remona transformations}.
\newblock 2011.

\bibitem[Dol12]{Dolgachev}
I.~Dolgachev.
\newblock {\em {C}lassical {A}lgebraic {G}eometry: a modern view}.
\newblock to be published in Cambridge Univ. Press, 2012.

\bibitem[Ful69]{Ful}
William Fulton.
\newblock {\em Algebraic curves. {A}n introduction to algebraic geometry}.
\newblock W. A. Benjamin, Inc., New York-Amsterdam, 1969.
\newblock Notes written with the collaboration of Richard Weiss, Mathematics
  Lecture Notes Series.

\bibitem[GP78]{GPI}
Laurent Gruson and Christian Peskine.
\newblock Genre des courbes de l'espace projectif.
\newblock In {\em Algebraic geometry ({P}roc. {S}ympos., {U}niv. {T}roms\o,
  {T}roms\o, 1977)}, volume 687 of {\em Lecture Notes in Math.}, pages 31--59.
  Springer, Berlin, 1978.

\bibitem[Guf04]{Guf}
S{\'e}bastien Guffroy.
\newblock Irr\'eductibilit\'e de {$\mathcal{H}_{d,g}$} pour {$d\leq11$} et
  {$g\leq2d-9$}.
\newblock {\em Comm. Algebra}, 32(12):4543--4558, 2004.

\bibitem[Har77]{Har}
Robin Hartshorne.
\newblock {\em Algebraic geometry}.
\newblock Springer-Verlag, New York, 1977.
\newblock Graduate Texts in Mathematics, No. 52.

\bibitem[IN04]{IN}
Yuji Ishii and Noboru Nakayama.
\newblock Classification of normal quartic surfaces with irrational
  singularities.
\newblock {\em J. Math. Soc. Japan}, 56(3):941--965, 2004.

\bibitem[IP99]{IP}
V.~A. Iskovskikh and Y.~G. Prokhorov.
\newblock {\em Algebraic geometry. {V}}, volume~47 of {\em Encyclopaedia of
  Mathematical Sciences}.
\newblock Springer-Verlag, Berlin, 1999.
\newblock Fano varieties.

\bibitem[Isk78]{Isk}
V.~A. Iskovskih.
\newblock Fano threefolds. {II}.
\newblock {\em Izv. Akad. Nauk SSSR Ser. Mat.}, 42(3):506--549, 1978.

\bibitem[Isk87]{Isko-duke}
V.~A. Iskovskikh.
\newblock On the rationality problem for conic bundles.
\newblock {\em Duke Math. J.}, 54(2):271--294, 1987.

\bibitem[JPR05]{JPR}
Priska Jahnke, Thomas Peternell, and Ivo Radloff.
\newblock Threefolds with big and nef anticanonical bundles. {I}.
\newblock {\em Math. Ann.}, 333(3):569--631, 2005.

\bibitem[JPR11]{JPR2}
Priska Jahnke, Thomas Peternell, and Ivo Radloff.
\newblock Threefolds with big and nef anticanonical bundles {II}.
\newblock {\em Cent. Eur. J. Math.}, 9(3):449--488, 2011.

\bibitem[Kal12]{Kal}
Anne-Sophie Kaloghiros.
\newblock A classification of terminal quartic 3-folds and applications to
  rationality questions.
\newblock {\em To appear in Mathematische Annalen}, 2012.

\bibitem[Kat87]{Katz}
Sheldon Katz.
\newblock The cubo-cubic transformation of {${\bf P}^3$} is very special.
\newblock {\em Math. Z.}, 195(2):255--257, 1987.

\bibitem[KM98]{KM}
J{\'a}nos Koll{\'a}r and Shigefumi Mori.
\newblock {\em Birational geometry of algebraic varieties}, volume 134 of {\em
  Cambridge Tracts in Mathematics}.
\newblock Cambridge University Press, Cambridge, 1998.
\newblock With the collaboration of C. H. Clemens and A. Corti, Translated from
  the 1998 Japanese original.

\bibitem[KPZ11]{KPZ}
Takashi Kishimoto, Yuri Prokhorov, and Mikhail Zaidenberg.
\newblock Affine cones over fano threefolds and additive group actions.
\newblock {\em preprint arXiv:1106.1312}, 2011.

\bibitem[LB82]{Barz}
Patrick Le~Barz.
\newblock Formules multis\'ecantes pour les courbes gauches quelconques.
\newblock In {\em Enumerative geometry and classical algebraic geometry
  ({N}ice, 1981)}, volume~24 of {\em Progr. Math.}, pages 165--197.
  Birkh\"auser Boston, Mass., 1982.

\bibitem[MM83]{MM}
Shigefumi Mori and Shigeru Mukai.
\newblock On {F}ano {$3$}-folds with {$B_{2}\geq 2$}.
\newblock In {\em Algebraic varieties and analytic varieties ({T}okyo, 1981)},
  volume~1 of {\em Adv. Stud. Pure Math.}, pages 101--129. North-Holland,
  Amsterdam, 1983.

\bibitem[Mor84]{Mori84}
Shigefumi Mori.
\newblock On degrees and genera of curves on smooth quartic surfaces in {${\bf
  P}^3$}.
\newblock {\em Nagoya Math. J.}, 96:127--132, 1984.

\bibitem[Nas08]{nasu}
Hirokazu Nasu.
\newblock The {H}ilbert scheme of space curves of degree {$d$} and genus
  {$3d-18$}.
\newblock {\em Comm. Algebra}, 36(11):4163--4185, 2008.

\bibitem[Pan10]{PanOld}
Ivan Pan.
\newblock On the untwisting of general de {J}onqui\`eres and cubo-cubic
  {C}remona transformations of $\mathbb{P}^3$.
\newblock {\em Preprint}, 2010.

\bibitem[Pan11]{pan:2011}
Ivan Pan.
\newblock On elementary links from $\mathbb{P}^3$ to a fano 3-fold and
  {C}remona transformations which factorize in a minimal form.
\newblock {\em Preprint}, 2011.

\bibitem[PS74]{PS}
C.~Peskine and L.~Szpiro.
\newblock Liaison des vari\'et\'es alg\'ebriques. {I}.
\newblock {\em Invent. Math.}, 26:271--302, 1974.

\bibitem[SD74]{SD}
B.~Saint-Donat.
\newblock Projective models of {$K-3$} surfaces.
\newblock {\em Amer. J. Math.}, 96:602--639, 1974.

\bibitem[Shi89]{shin}
Kil-Ho Shin.
\newblock {$3$}-dimensional {F}ano varieties with canonical singularities.
\newblock {\em Tokyo J. Math.}, 12(2):375--385, 1989.

\bibitem[SR49]{SR}
J.~G. Semple and L.~Roth.
\newblock {\em Introduction to {A}lgebraic {G}eometry}.
\newblock Oxford, at the Clarendon Press, 1949.

\bibitem[Tak89]{T}
Kiyohiko Takeuchi.
\newblock Some birational maps of {F}ano {$3$}-folds.
\newblock {\em Compositio Math.}, 71(3):265--283, 1989.

\end{thebibliography}

\end{document}